\documentclass[oneside,english]{amsart}
\usepackage[T1]{fontenc}
\usepackage[a4paper]{geometry}
\usepackage{color}
\usepackage{amstext}
\usepackage{amsthm}
\usepackage{amssymb}

\usepackage{graphicx}
\makeatletter
\numberwithin{equation}{section}
\numberwithin{figure}{section}
\theoremstyle{plain}
\newtheorem{thm}{\protect\theoremname}[section]
\theoremstyle{plain}
\newtheorem{assumption}[thm]{\protect\assumptionname}
\theoremstyle{plain}
\newtheorem{lem}[thm]{\protect\lemmaname}
\theoremstyle{definition}
\newtheorem{example}[thm]{\protect\examplename}
\theoremstyle{plain}
\newtheorem{cor}[thm]{\protect\corollaryname}
\theoremstyle{plain}
\newtheorem{prop}[thm]{\protect\propositionname}
\theoremstyle{definition}
\newtheorem{defn}[thm]{\protect\definitionname}
\theoremstyle{definition}
\newtheorem{rem}[thm]{\protect\remarkname}

\makeatletter
\numberwithin{equation}{section}
\numberwithin{figure}{section}

\usepackage{cite}

\usepackage[all]{xy}

\renewcommand{\div}{\operatorname{div}}
\newcommand{\loc}{\operatorname{loc}}
\newcommand{\sgn}{\operatorname{sgn}}

\usepackage{hyperref}

\numberwithin{equation}{section}

\newcommand{\Bcal} {{\mathcal B}}

\newcommand{\Ocal} {{\mathcal O}}

\newcommand{\R}{\mathbb{R}}
\newcommand{\N}{\mathbb{N}}

\renewcommand{\P}{\mathbb{P}}
\newcommand{\E}{\mathbb{E}}

\newcommand{\dualdel}[3]{\sideset{_{#1^\ast}}{_{#1}}{\mathop{\langle{#2},{#3}\rangle}}}
\newcommand{\sprod}[2]{\left\langle #1 , #2 \right\rangle}
\newcommand{\norm}[1]{\left\lVert #1  \right\rVert}
\newcommand{\abs}[1]{\left\lvert #1 \right\rvert}

\newcommand{\BIGOP}[1]{\mathop{\mathchoice%
{\raise-0.22em\hbox{\huge $#1$}}%
{\raise-0.05em\hbox{\Large $#1$}}{\hbox{\large $#1$}}{#1}}}

\newcommand{\BIGboxplus}{\mathop{\mathchoice%
{\raise-0.35em\hbox{\huge $\boxplus$}}%
{\raise-0.15em\hbox{\Large $\boxplus$}}{\hbox{\large $\boxplus$}}{\boxplus}}}


\makeatother

\AtBeginDocument{
  
}

\makeatother

\usepackage{babel}
\providecommand{\assumptionname}{Assumption}
\providecommand{\corollaryname}{Corollary}
\providecommand{\definitionname}{Definition}
\providecommand{\examplename}{Example}
\providecommand{\lemmaname}{Lemma}
\providecommand{\propositionname}{Proposition}
\providecommand{\theoremname}{Theorem}
\providecommand{\remarkname}{Remark}

\begin{document}
\author{Florian Seib}
\author{Wilhelm Stannat}
\address{Technische Universit\"{a}t Berlin (MA 7-2)\\
Institut f\"{u}r Mathematik\\
Stra\ss{}e des 17. Juni 136\\
10623 Berlin\\
Germany}
\email{florian.seib@gmx.de, stannat@math.tu-berlin.de}
\author{Jonas M. T\"{o}lle}
\address{Aalto University\\
Department of Mathematics and Systems Analysis\\
PO Box 11100 (Otakaari~1, Espoo)\\
00076 Aalto\\
Finland}
\email{jonas.tolle@aalto.fi}

\date{\today}
\title[Stability and moment estimates for stochastic $\Phi$-Laplace equations]{Stability and moment estimates for the stochastic singular $\Phi$-Laplace
equation}
\begin{abstract}
We study stability, long-time behavior and moment estimates for stochastic
evolution equations with additive Wiener noise and with singular drift
given by a divergence type quasilinear diffusion operator which may
not necessarily exhibit a homogeneous diffusivity. Our results cover
the singular stochastic $p$-Laplace equations and, more generally, singular stochastic $\Phi$-Laplace
equations with zero Dirichlet boundary conditions. We obtain improved
moment estimates and quantitative convergence rates of the ergodic
semigroup to the unique invariant measure, classified in a systematic
way according to the degree of local degeneracy of the potential at
the origin. We obtain new concentration results for the invariant
measure and establish maximal dissipativity of the associated Kolmogorov
operator. In particular, we recover the results for the curve shortening
flow in the plane by Es-Sarhir, von Renesse and Stannat, NoDEA 16(9), 
2012, \cite{ESvRS}, and improve the results by Liu and T\"olle, ECP 16, 2011, \cite{LiuToelle}.
\end{abstract}


\thanks{This work is licensed under the Creative Commons Attribution 4.0 International License. To view a copy of this license, visit \url{http://creativecommons.org/licenses/by/4.0/} or send a letter to Creative Commons, PO Box 1866, Mountain View, CA 94042, USA. \includegraphics[height=1em]{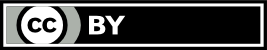}
Original work published in Journal of Differential Equations 377 (2023) 663--693, \url{https://doi.org/10.1016/j.jde.2023.09.019}.}

\maketitle

\section{Introduction}

The study of long-time behavior for Markovian dynamics associated
to stochastic evolution equations has a long history. Its significance
to mathematical physics is emphasized in the equations of Fokker and
Planck, and Kolmogorov \cite{Kolmogorov}, which are discussed in
detail in the monographs \cite{BogachevKrylovRoecknerShaposhnikov,Pavliotis}.
Extended to stochastic partial differential equations (SPDEs), the
analysis of invariant probability distributions and Kolmogorov operators
in infinite dimensions \cite{daprato2012kolmogorov,DaPrZergodicity}
provides improved insight into the averaged dynamics of PDEs perturbed
by random noise with infinitely many frequency modes \cite{DaPrZ2nd}.

We are interested in stability and moment estimates for nonlinear
SPDEs of the following type
\begin{equation}
\partial_{t}X_{t}=\div[\phi(\nabla X_{t})]+B\,\partial_{t}W_{t},\quad X_t|_{\partial\Ocal}=0,\quad t>0,\quad X_{0}=u_{0},\label{eq:main-spde}
\end{equation}
where $u_{0}\in L^{2}(\Ocal)$ for a bounded, open and convex domain $\Ocal\subset\R^{d}$ with piecewise
$C^{2}$-boundary,
$d\ge1$. We require that the solution $X_{t}$, $t>0$ vanishes on
the boundary $\partial\Ocal$, that is, we impose zero Dirichlet boundary
conditions. Here, $B$ is a linear
operator that is regularizing the spatial component and $\{W_{t}\}_{t\ge0}$ a cylindrical Wiener process on
some filtered probability space $(\Omega,\mathcal{F},\{\mathcal{F}_{t}\}_{t\ge0},\mathbb{P})$ with values in some separable Hilbert space $U$.
The nonlinearity $\phi:\R^{d}\to\R^{d}$ is monotone in the sense
that
\begin{equation}\label{eq:monotone1}
\langle\phi(x)-\phi(y),x-y\rangle\ge0,\quad x,y\in\R^{d}.
\end{equation}
Existence and uniqueness of solutions to SPDEs with nonlinear monotone
drift were first systematically studied in \cite{KrRo0,Pard}.

In this work, we shall study the singular drift case, including examples where the Jacobian of $\phi$ has a singularity in zero. We assume that
$\phi$ has at most linear growth, i.e. for some $C>0$,
\begin{equation}\label{eq:linear_growth}|\phi(x)|\le C(1+|x|).\end{equation}
We refer to \cite{GessToelle14,LiuRoe,PrevotRoeckner}
for a well-posedness theory for such equations. Examples of such $\phi$
that satisfy our main Assumption \ref{assu:C}, include, among
others, the following ones.

\subsubsection*{Stochastic singular $p$-Laplace}

For $p\in(1,2)$, let $\phi(x)=|x|^{p-2}x$ and yield the stochastic
singular $p$-Laplace equation. It has been first addressed in \cite{Liu09pLaplace}.
The results have been extended in \cite{LiuRoe,GessToelle14,GessToelle15,GessRoe2017}.
The multi-valued stochastic $1$-Laplace equation, which is also known
as the \emph{total variation flow} is not within the scope of this
work, instead, we refer to \cite{VBMR,GessRoe2017,GessToelle14,GessToelle16,GessToelle15}.
The $L^{1}$-theory of the stochastic $p$-Laplace equation has been
discussed in \cite{SapountzoglouZimmermann2019}.

\subsubsection*{Stochastic non-Newtonian fluids}

For $p\in(1,2)$, let $\phi(x)=\left(1+|x|^{2}\right)^{\frac{p-2}{2}}x$,
which yields a simplified model for stochastic singular power law
fluids without convection term. It exhibits a similar nonlinearity
as in an equation modeling non-Newtonian fluids with stochastic perturbation
in the shear-thinning case, see \cite{Breit,Sauer,BR15}.

\subsubsection*{Stochastic curve shortening flow}

In one spatial dimension $d=1$, we may formally differentiate $\phi(x)=\arctan(x)$
in order to obtain the so-called (additive noise) \emph{stochastic
curve shortening flow}
\[
\partial_{t}X_{t}=\frac{\partial_{xx}X_{t}}{1+(\partial_{x}X_{t})^{2}}+B\,\partial_{t}W_{t},\quad t>0,\quad X_{0}=u_{0},
\]
which is a formulation of the stochastic mean curvature flow in the
plane, see \cite{ESvR,GessRoe2017}. It has been generalized toward
two spatial dimensions in \cite{HRvR2017}.

\subsubsection*{Further examples}

Furthermore, we have the minimal surface flow $\phi(x)=\frac{x}{\sqrt{1+|x|^{2}}}$,
which is related to the mean curvature flow, see \cite{Funaki,BloemkerRomito},
and the logarithmic diffusion $\phi(x)=\log(1+|x|)\frac{x}{|x|}$,
as addressed e.g. in \cite{GessToelle14}.

The main problem with these kinds of nonlinear diffusivities remains
the lack of coercivity and dissipativity in the general situation.
This may be achieved by Sobolev embeddings in certain cases \cite{Liu13},
however, energy methods as e.g. stochastic variational inequalities
\cite{GessToelle15,GessRoe2017,ScarpaStefanelli2021,VBMR} seem to
deal with well-posedness issues better without loosing for instance
the Markov property of the semigroup \cite{GessToelle16}.

Other variational approaches to divergence type drift SPDEs can be
found in \cite{MarinelliScarpa2018monotone,MarinelliScarpa2018,Liu13,VBMR,LiuRoe}.
The case of gradient type Stratonovich noise is addressed in \cite{CiotirToelle,Toelle,BBHT,BBT16}.

\subsection*{Long-time behavior}

Even in the case of analytically weak variational solutions, we may
define a \emph{Feller semigroup} $\{P_{t}\}_{t\ge0}$ associated to
the dynamics of \eqref{eq:main-spde} via extension of
\[
P_{t}f(u):=\E[f(X_{t}^{u})],\quad f\in\operatorname{Lip}_{b}(L^{2}(\Ocal)),\quad t\ge0,
\]
where $\{X_{t}^{u}\}_{t\ge0}$ denotes a solution to \eqref{eq:main-spde}
with $X_{0}=u$, see \cite{DaPrZergodicity}. We call a Borel probability measure $\mu$
on $L^{2}(\Ocal)$ \emph{invariant} with respect to $\{P_{t}\}_{t\ge0}$
if
\[
\int_{L^{2}(\Ocal)}P_{t}f\,d\mu=\int_{L^{2}(\Ocal)}f\,d\mu
\]
for every $t\ge0$ and every $f\in C_{b}(L^{2}(\Ocal))$. We refer
to \cite{Stannat2011} for an overview article on this topic. Uniqueness
of invariant measures for singular stochastic equations in low spatial
dimensions was first proved with methods from \cite{KPS} for the
stochastic curve shortening flow in \cite{ESvR}. These results were
then extended to the stochastic singular $p$-Laplace equation in
\cite{LiuToelle} and to more general equations in \cite{GessToelle14,Neuss2021}.
In particular, for additive noise, the decay behavior of deterministic
solutions (that is, for $B\equiv0$) plays an important role, see
\cite{Porzio2009,Porzio2011,DiBenedetto:1993kl} for results of this
type. We would like to point out that our result is novel even for the deterministic case. Second order estimates for the
deterministic parabolic $p$-Laplace system have been obtained in \cite{CM19,CM20}.
Ergodicity of the semigroup and uniqueness of invariant measures
for local and nonlocal singular $p$-Laplace equations in all spatial
dimensions was proved in \cite{GessToelle16}.

We are also interested in improved moment estimates for singular monotone
drift SPDEs which yield more explicit convergence rates in the spirit
of \cite{ESS2010,ESvRS,LiuToelle}. To this end, we may assume that
the singularity is of a certain polynomial type, more precisely, we shall assume that $\phi$ has a radially symmetric potential, that is, $\phi:\R^{d}\to\R^{d}$ is of the form
\begin{equation}
\phi(z)=\begin{cases}
\psi(|z|)\frac{z}{|z|}, & \text{if}\quad z\not=0,\\
0, & \text{if}\quad z=0,
\end{cases}\label{eq:phi_form}
\end{equation}
where $\psi:\R\to\R$ satisfies the following
\begin{assumption}
\label{assu:C} Suppose that $\psi\in W_{\loc}^{1,1}(\R)$ such that
there exists $C>0$ with
\begin{enumerate}
\item[(C1)] $\psi'(r)>0$ for every $r>0$, 
\item[(C2)] $\psi(r)=-\psi(-r)$ for every $r\in\R$,
\item[(C3)] $\psi(r)\le C(1+|r|)$ for every $r\ge0$,
\item[(C4)]  there exist $a\ge0$, $b>0$, $s\in(0,2]$ with 
\[
(a+b|r|^{s})|\psi'(r)|\ge1
\]
 for every $r\ge0$, 
\item[(C5)] there exist $c>0$ and $K\ge0$ such that
\[
\psi(r)\cdot r\ge c|r|-K,\quad r\ge0.
\]
\item[(C6)] $$0<s\le\frac{4}{d}\wedge 2,$$
\end{enumerate}
\end{assumption}

We are ready to formulate our main result, see Theorems \ref{support} and \ref{thm:semigroup}, which generalizes the results of \cite{ESS2010,ESvRS,LiuToelle,GessToelle14}, which assume the more restrictive assumption 
$s<\frac{4}{d+2}\wedge 1$ compared to (C6).

\begin{thm}\label{thm:main_in_intro}
Assume (C1)--(C6) and that\footnote{We denote the \emph{space of Hilbert-Schmidt
operators} from the separable Hilbert space $U$ to a separable Hilbert space $S$ by $L_{2}(U,S)$.} $B\in L_2(U,H_0^1(\Ocal))$. Set
\[s^{\ast}=(2-s)\vee\frac{4-s}{2+s}.\]
\begin{enumerate}
\item Then there exists a unique invariant Borel probability
measure $\mu$ of $\{P_t\}_{t\ge 0}$ that is concentrated on the subset $W^{2,\alpha}_0(\Ocal)\cap H^1_0(\Ocal)\subset L^2(\Ocal)$ for any $1<\alpha\leq d\frac{2-s}{d-s}$ if $d\ge 2$ and $\alpha=1$ if $d=1$ with
\[
\int_{L^2(\Ocal)}\|u\|_{W_{0}^{2,\alpha}(\Ocal)}^{s^{\ast}}\,\mu(du)+\int_{L^{2}(\Ocal)}\lVert u\rVert_{H_0^1(\Ocal)}^{s^{\ast}}\,\mu(du)<+\infty.
\]
\item
Let $d_{0}:=1\vee\frac{d}{2}$ and $$0<\beta\leq\beta^{\star}:=1\wedge\frac{8-2s}{s(2+s)d_{0}}\in\left[\frac{1}{2d_{0}},1\right].$$
Then there exist non-negative constants $C_{1}$, $C_{2}$
and $C_{3}$ such that
\[
\limsup_{t\to+\infty}\left[t^{\frac{s^{\ast}}{s}}\frac{\lvert P_{t}f(u)-P_{t}f(v)\rvert}{\lVert u-v\rVert_{L^2(\Ocal)}^{\beta}}\right]\leq C_{1}|f|_{\beta}\left(C_{2}+C_{3}\int_{L^2(\Ocal)}\lVert u\rVert_{H_0^1(\Ocal)}^{s^{\ast}}\,\mu(du)\right)
\]
for every $u,v\in H_0^1(\Ocal)$ and for every $f:L^2(\Ocal)\to\R$ that is bounded and $\beta$-H\"{o}lder-continuous\footnote{$f:L^2(\Ocal)\to\R$ is called \emph{$\beta$-H\"{o}lder-continuous} if
$\sup_{u\not=v}\frac{\lvert f(u)-f(v)\rvert}{\lVert u-v\rVert_{L^2(\Ocal)}^{\beta}}=:\lvert f\rvert_{\beta}<+\infty$.}.
\end{enumerate}
\end{thm}

Finally, we shall address the maximal dissipativity of the associated
Kolmogorov operator with infinitely many variables, see \cite{Stannat2011NSC,BarbuDaPrato2010,daprato2012kolmogorov,ESS2009,ESvRS,Stannat2011,BogachevDaPratoRoeckner,MarinelliScarpa2020}
for related works.

We note that the degenerate $p$-Laplace, i.e. the case $p>2$, may be treated more easily
with the Krylov-Bogoliubov method \cite{DaPrZergodicity}, as is done
in the variational setup in \cite{BarbuDaPrato2006}. Previous approaches
include dimension-free Harnack inequalities \cite{Liu09harnack,HongLiLiu,LiuWang}
and the more direct approach via irreducibility \cite{Zhang2019}.
In the case of multiplicative noise, existence of invariant measures
for monotone drift equations has been proved in \cite{ESST}.

\subsection*{Structure of the paper}

We start with collecting some preliminary results in Section \ref{sec:Main-assumptions}
including a known well-posedness result for \eqref{eq:main-spde} and results on invariant measures of Markovian Feller semigroups. We conclude Section \ref{sec:Main-assumptions} with a discussion of our assumptions and a collection of our main examples for $\phi$. In
Section \ref{sec:Stability}, as one the main results, we shall prove
a stability result for the semigroup, see Proposition \ref{thm:decay}.
In Section \ref{sec:Moment-estimates}, we shall first record a result for the case of one
spatial dimension, see Lemma \ref{lem:a-priori}. Next, we prove the improved moment estimates and a concentration
result for the unique invariant measure, see Theorem \ref{support}, which implies our main result Theorem \ref{thm:main_in_intro} (i).
In Section \ref{sec:decayandexamples}, the main result Theorem \ref{thm:main_in_intro} (ii) is proved and applied to the main
examples, see Theorem \ref{thm:semigroup} and the subsequent examples. In Section \ref{sec:Maximal-dissipativity}, under an additional assumption, we prove the
maximal dissipativity of the Kolmogorov operator associated to the
semigroup, see Theorem \ref{thm:max}.

\section{\label{sec:Main-assumptions}Preliminary results}

\subsection{Existence and uniqueness of solutions}

Let us recall the following conditions from \cite{GessToelle14},
simplified with regard to the time-dependence of the drift coefficients,
which is not needed here. The results are adapted to the Gelfand triple\footnote{Let $H$ be a separable Hilbert space and let $V$ be another Hilbert
space (with topological dual $V^{\ast}$) which is embedded densely
and continuously into $H$. Then, $H$ is identified with its dual
via the Riesz isometry. The triple $V\subset H\subset V^{\ast}$ is
called \emph{Gelfand triple}.}
\begin{equation}\label{eq:Gelfand}
V:=H_{0}^{1}(\Ocal)\subset H:=L^{2}(\Ocal)\subset V^{\ast}=H^{-1}(\Ocal)
\end{equation}
that also encodes the zero Dirichlet boundary condition. Here, we use the shorthand notation $H_0^1(\Ocal)=W^{1,2}_0(\Ocal)$ and the notation $H^{-1}(\Ocal)$ for
the topological dual space of $H_0^1(\Ocal)$.

Suppose that $A:V\to V^{\ast}$ satisfies the
following conditions: There exists a constant $C>0$ such that
\begin{enumerate}
\item[(A1)]  The map $u\mapsto A(u)$ is maximal monotone\footnote{That is, the graph of $A$ is maximal in the class of monotone graphs,
ordered by set-inclusion, see \cite{barbubook} for details.}.
\item[(A2)]  For all $u\in V$:
\[
\|A(u)\|_{V^{\ast}}\le C\|u\|_{V}.
\]
\item[(A3)]  For all $u\in V$, and for all $n\in\N$:
\[
2\dualdel{V}{A(u)}{T_{n}(u)}\ge-C\|u\|_{V}^{2},
\]
such that $C$ is independent of $n$, where $T_{n}:=n\left(\operatorname{Id}-(\operatorname{Id}-\frac{\Delta}{n})^{-1}\right)$
denotes the Yosida-approximation of the negative Dirichlet Laplace
operator $-\Delta=-\sum_{i=1}^{d}\partial_{x_{i}}\partial_{x_{i}}$.
\end{enumerate}
Let $U$ be a separable Hilbert space. For a separable Hilbert space $S$, denote the \emph{space of Hilbert-Schmidt
operators} from $U$ to $S$ by $L_{2}(U,S)$. Denote by $\{W_{t}\}_{t\ge0}$ a cylindrical Wiener process\footnote{See \cite{DaPrZ2nd} for the notion of a cylindrical Wiener process
in Hilbert space.} in $U$
for a stochastic basis $(\Omega,\mathcal{F},\{\mathcal{F}_{t}\}_{t\ge0},\mathbb{P})$ that satisfies the standard assumptions.

\begin{assumption}\label{assu:B}Assume that
\begin{enumerate}
\item[(B1)] $B\in L_2(U,V)$.
\end{enumerate}
\end{assumption}

\begin{defn}
\label{def:sln-GT} We say that a continuous $\{\mathcal{F}_{t}\}_{t\ge0}$-adapted
stochastic process $X:[0,T]\times\Omega\to H$ is a \emph{solution
}to
\begin{equation}
dX_{t}+A(X_{t})\,dt= B\,dW_{t},\quad X_{0}=u_{0},\label{eq:Ito-multivalued}
\end{equation}
if $X\in L^{2}(\Omega;C([0,T];H))\cap L^{2}([0,T]\times\Omega;V)$
and solves the following integral equation in $V^{\ast}$
\[
X_{t}=u_{0}-\int_{0}^{t}A(X_\tau)\,d\tau+B\,W_{t},
\]
$\mathbb{P}$-a.s. for all $t\in[0,T]$.
\end{defn}

\begin{thm}
\label{thm:sln-ex-GT} Suppose that conditions (A1)--(A3), (B1)
hold. Let $u_{0}\in L^{2}(\Omega,\mathcal{F}_{0},\mathbb{P};V)$.
Then there exists a unique solution in the sense of the previous definition
to the equation
\begin{equation}
dX_{t}+A(X_{t})\,dt=B\,dW_{t},\quad X_{0}=u_{0},\label{eq:abstract-Ito-eq-solution}
\end{equation}
that satisfies
\[
\E\left[\sup_{t\in[0,T]}\|X_{t}\|_{V}^{2}\right]<+\infty.
\]
\end{thm}

\begin{proof}
See \cite[Theorem 4.4]{GessToelle14}. 
\end{proof}
\begin{defn}
\label{def:limit-solution}An $\{\mathcal{F}_{t}\}_{t\ge0}$-adapted
stochastic process $X\in L^{2}(\Omega;C([0,T];H))$ is called a \emph{generalized
solution}\footnote{Also referred to as \emph{limit solution}.}\emph{
}to \eqref{eq:Ito-multivalued} with starting point $u_{0}\in H$
if for all approximations $u_{0}^{m}\in V$, $m\in\mathbb{N}$ with
$\|u_{0}^{m}-u_{0}\|_{H}\to0$ as $m\to+\infty$ and all $B_{m}\in L_{2}(U,V)$
such that $B_{m}\to B\in L_2(U,H)$ strongly
in $L_{2}(U,H)$ as $m\to+\infty$, we
have that
\[
X^{m}\to X\quad\text{strongly in\ensuremath{\quad}\ensuremath{L^{2}(\Omega;C([0,T];H))}\ensuremath{\quad}as\ensuremath{\quad}\ensuremath{m\to+\infty}.}
\]
\end{defn}

\begin{thm}
\label{thm:ex-limit-sln} Suppose that conditions (A1)--(A3) hold and that $B\in L_2(U,H)$.
Let $u_{0}\in L^{2}(\Omega,\mathcal{F}_{0},\mathbb{P};H)$.
Then there exists a unique generalized solution in the sense of the
previous definition to the equation
\begin{equation}
dX_{t}+A(X_{t})\,dt=B\,dW_{t},\quad X_{0}=u_{0}.\label{eq:abstract-Ito-eq-limit-solution}
\end{equation}
\end{thm}

\begin{proof}
See \cite[Theorem 4.6]{GessToelle14}.
\end{proof}

On a bounded, open and convex domain $\Ocal\subset\R^{d}$ with piecewise
$C^{2}$-boundary, we shall consider the $\Phi$-Laplace SPDE
\begin{equation}
dX_{t}=\div[\phi(\nabla X_{t})]\,dt+B\,dW_{t},\quad X_{0}=u_{0},\label{eq:main_Ito_SPDE}
\end{equation}
with zero Dirichlet boundary conditions for $t\in[0,T]$ and $u_{0}\in H:=L^{2}(\Ocal)$,
where $\phi$ is of the form \eqref{eq:phi_form}. We sometimes use
the notation
\[
A(u):=-\div[\phi(\nabla u)],
\]
which is made rigorous in the compact Gelfand triple \eqref{eq:Gelfand}. Indeed,
\begin{equation}
\dualdel{V}{A(u)}{v}=\int_{\mathcal{O}}\langle\phi(\nabla u),\nabla v\rangle\,dx,\quad u,v\in V.\label{eq:A_defi}
\end{equation}

\begin{lem}
Suppose that $\phi$ satisfies (C1)--(C3). Then $A$ satisfies conditions
(A1)--(A3).
\end{lem}
\begin{proof}
Set
\[
\Psi(x):=\int_{0}^{|x|}\psi(r)\,dr,
\]
which is a radially symmetric, convex, continuous and convex function
with $\Psi(0)=0$ and at most quadratic growth. We have the representation
\[
A(u)=\nabla^{\ast}\circ\partial\left(\int_{\Ocal}\Psi(\cdot)\,dx\right)\circ\nabla u,\quad u\in V,
\]
where $\partial$ denotes the subgradient and $\nabla^{\ast}=(-\Delta)^{-1}\circ\div$
denotes the adjoint operator of $\nabla:V=H_{0}^{1}(\Ocal)\to L^{2}(\Ocal;\R^{d})$.
We may apply the results from \cite[Proposition 7.1 and Section 7.2.1]{GessToelle14}
to yield the claim. See, in particular, \cite[Example 7.9]{GessToelle14}.
We note that we use the convexity assumption on the boundary here,
cf. \cite[Proposition D.2]{GessToelle14} and \cite[Proposition 8.2]{VBMR}.
\end{proof}
Applying Theorem \ref{thm:ex-limit-sln}, we obtain the following
existence and uniqueness result.

\begin{cor}
Suppose that $\phi$ satisfies (C1)--(C3), (C5) and that $B$ satisfies
(B1). Then there exists a unique
generalized solution to \eqref{eq:main_Ito_SPDE} in the sense of
Definition \ref{def:limit-solution}.
\end{cor}

\subsection{Feller semigroups and invariant measures}

Following \cite{DaPrZergodicity}, we set
\begin{equation}
P_{t}f(u):=\E[f(X_{t}^{u})],\quad f\in\operatorname{Lip}_{b}(H),\quad t\ge0,\label{eq:defi_semigroup}
\end{equation}
where $\{X_{t}^{u}\}_{t\ge0}$ denotes a (generalized) solution to
\eqref{eq:main_Ito_SPDE} with initial datum $X_{0}=u$. By the results
in \cite[Section 6.4]{GessToelle14}, the extension of $\{P_{t}\}_{t\ge0}$
to $C_{b}(H)$ defines a Markovian transition semigroup. It can
easily be seen that $\{P_{t}\}_{t\ge0}$ is stochastically continuous, that is,
\[\lim_{t\to 0+}P_t f(u)=f(u),\quad\text{for all}\;\;f\in\operatorname{Lip}_{b}(H),\;\;u\in H,\]
cf. \cite[Proposition 2.1.1]{DaPrZergodicity},
and satisfies the Feller property, that is, for any $f\in C_b(H)$ and $t\ge0$ one has $P_t f\in C_b(H)$.
Denote its dual semigroup restricted to finite Borel measures by $\{P^\ast_{t}\}_{t\ge0}$. 
A probability measure $\mu$ is called \emph{invariant} for $\{P_{t}\}_{t\ge0}$ if $P_t^\ast \mu=\mu$ for any $t\ge 0$.

We recall the following concepts defined e.g. in \cite{KPS}.
\begin{defn}
We say that $\{P_{t}\}_{t\ge0}$ is \emph{weak-$\ast$-mean
ergodic} if there exists a Borel probability measure $\mu_\ast$ on $\Bcal(H)$, that is the Borel sets of $H$, such that
\begin{equation}\label{weakstarergodic}\text{w-}\lim_{T\to+\infty}\frac{1}{T}\int_0^T P_t^\ast\nu\,dt=\mu_\ast\quad\text{for every Borel probability measure $\nu$ on $\Bcal(H)$,}\end{equation}
where the limit is in the sense of weak convergence of probability measures.

We say that the \emph{weak law of large numbers} holds for $\{P_{t}\}_{t\ge0}$, for a function $f\in\operatorname{Lip}_{b}(H)$ and for a probability measure $\nu$ on $\Bcal(H)$ if
\[\P_\nu\text{-}\lim_{T\to+\infty}\frac{1}{T}\int_0^T f(X_t)\,dt=\int_H f\,d\mu_\ast,\]
where $\mu_\ast$ denotes the invariant measure of $\{P_{t}\}_{t\ge0}$ and $\{X_t\}_{t\ge 0}$ denotes the Markov process related to $\{P_{t}\}_{t\ge0}$ whose initial distribution is $\nu$ and whose path measure is $\P_\nu$, and where the convergence takes place in $\P_\nu$-probability.
\end{defn}

As noted in \cite[Remark 3]{KPS}, \eqref{weakstarergodic} implies uniqueness of the invariant measure.

\begin{thm}
Suppose that $\phi$ satisfies (C1)--(C5) and that
\begin{equation}\label{eq:coercivitybound}
s<\frac{4}{d+2}\wedge 1.
\end{equation}
Then there exists a unique invariant measure $\mu$ for the semigroup
$\{P_{t}\}_{t\ge0}$ such that $\{P_{t}\}_{t\ge0}$ is weak-$\ast$-mean
ergodic and the weak law of large numbers holds for any $f\in\operatorname{Lip}_{b}(H)$ and any probability measure $\nu$.
\end{thm}

\begin{proof}
We would like to use \cite[Remark 5.5 and Proposition 7.1]{GessToelle14}
to see that the hypotheses of \cite[Theorem 5.6]{GessToelle14}
are satisfied. For that we only have to verify that our equation admits a Lyapunov function with compact sublevel sets and that the solution to the deterministic counterpart of our equation (that is, $B\equiv 0$) vanishes for $t\to+\infty$.

By Lemma \ref{lem:growth}, we get for $u\in V$,
\begin{equation}
2\dualdel{V}{A(u)}{u}=2\int_{\mathcal{O}}\langle\phi(\nabla u),\nabla u\rangle\,dx\ge2\int_{\Ocal}|\nabla u|^{2-s}\,dx-2K|\Ocal|=:\Theta(u),\label{eq:det_bound}
\end{equation}
where we denote $|\Ocal|:=\int_{\Ocal}\,dx$. By \eqref{eq:coercivitybound}
and the Rellich-Kodrachov theorem, $W_{0}^{1,2-s}(\Ocal)\subset L^{2}(\Ocal)$
compactly, thus $\Theta$ is a Lyapunov function with compact sublevels.
Let $(v_{t})_{t\ge0}$ be a solution to the deterministic equation,
\[
dv_{t}+A(v_{t})\,dt=0,\quad t>0,\quad v_{0}=u.
\]
It remains to prove that
\[
\lim_{t\to+\infty}\|v_{t}\|_{H}=0
\]
for every initial datum $u\in H$. By the chain rule,
\[
\frac{d}{dt}\|v_{t}\|_{H}^{2}\le-2C\left(\|v_{t}\|_{H}^{2}\right)^{\frac{2-s}{2}}+2K|\Ocal|.
\]
In analogy to \cite[Proof of Remark 5.5]{GessToelle14}, we see that
$f(t):=\|v_{t}\|_{H}^{2}e^{-2K|\Ocal|t}$ is a subsolution to the
ordinary differential equation
\[
f^{'}(t)=-2Cf(t)^{\frac{2-s}{2}}.
\]
Hence
\[
\|v_{t}\|_{H}^{2}\le e^{2K|\Ocal|t}\left(\left(\|u\|_{H}^{s}-Cst\right)\vee0\right)^{\frac{2}{s}}\longrightarrow0,
\]
as $t\to+\infty$. Now, we may apply \cite[Remark 5.5 and Proposition 7.1]{GessToelle14}
and see that the hypotheses for \cite[Theorem 5.6]{GessToelle14}
are satisfied. Thus the claimed result follows.
\end{proof}

The above result is improved in Theorems \ref{support} and \ref{thm:decay} below, as we just need to assume (C6) instead of \eqref{eq:coercivitybound}.
We refer to \cite{BCR2019} for a new unified two-step approach to proving the existence of invariant measures.

\subsection{Discussion of the assumptions} 

If $s\in (0,1]$, (C5) follows readily from the other assumptions.
\begin{lem}
\label{lem:growth} Assumptions (C1)--(C4), together with the assumption
that (C4) holds with $s\in(0,1]$, imply that there exist constants
$c=c(a,b)>0$ and $K=K(a,b,s,C)\ge0$ such that
\[
\psi(r)\cdot r\ge c|r|^{2-s}-K,\quad r\ge0.
\]
In particular, then (C5) follows.
\end{lem}

\begin{proof}
Let $s\in(0,1]$. Integrating $(a+b|r|^{s})|\psi'(r)|\ge1$ by parts
over $[0,x]$, $x\ge0$ yields
\[
(a+b|x|^{s})\psi(x)\ge x+bs\int_{0}^{x}\psi(r)|r|^{s-1}\,dr\ge x.
\]
Hence,
\[
a|x|\psi(x)+a\psi(x)+b|x|\psi(x)\ge a|x|^{1-s}\psi(x)+b|x|\psi(x)\ge|x|^{2-s}.
\]
Using (C3) and rearranging terms yields
\[
K(a,s,C)+\frac{1}{2}|x|^{2-s}+(a+b)|x|\psi(x)\ge|x|^{2-s},
\]
which proves the claim for $s\in(0,1]$.
\end{proof}
In Section \ref{sec:decayandexamples}, we shall discuss the following examples.
\begin{example}
\label{exa:main_examples}The following examples satisfy Assumptions
(C1)--(C5).
\begin{description}
\item [{Singular$~$$p$-Laplace}] Let $p\in(1,2)$. Then $\psi(r)=|r|^{p-2}r$
satisfies (C1)--(C5) with $C=c=K=1$, $a=0$, $b=(p-1)^{-1}$, $s=2-p$.
\item [{Non-Newtonian~fluids}] Let $p\in(1,2)$. Then $\psi(r)=(1+|r|^{2})^{\frac{p-2}{2}}r$
satisfies (C1)--(C5) with $C=K=c=1$, $a=b=(p-1)^{-1}$, $s=2-p$.
\item [{$\log$-diffusion}] Let $\psi(r)=\log(1+|r|)\sgn(r)$. Then
$\psi$ satisfies (C1)--(C5) with $C=a=b=s=c=1$, $K=2$.
\item [{Minimal$~$surface$~$flow}] Let $\psi(r)=\frac{r}{\sqrt{1+|r|^{2}}}$.
Then $\psi$ satisfies (C1)--(C5) with $C=a=b=s=c=K=1$.
\item [{Curve$~$shortening$~$flow}] Let $\psi(r)=\arctan(r)$. Then $\psi$
satisfies (C1)--(C5) with $C=a=b=c=K=1$, $s=2$.
\end{description}
\end{example}

\begin{rem}\label{rem:improvement}
Note that for the first two examples, \eqref{eq:coercivitybound} translates to $p>\frac{2d}{d+2}$, and $d<\frac{2p}{2-p}$ respectively, which was assumed e.g. in \cite{LiuToelle}. However, for the first two examples, we just need to assume condition (C6) which is $p\ge 2-\frac{4}{d}$, and $d\le \frac{4}{2-p}$ respectively.
\end{rem}

\section{\label{sec:Stability}Stability}

In this section, stability, that is, rates of convergence for large times for the solutions to \eqref{eq:main_Ito_SPDE} starting at two distinct initial data
will be established.
Unless otherwise stated, we assume that conditions (C1)--(C6) hold.
On a bounded, open and convex domain $\Ocal\subset\R^{d}$ with piecewise
$C^{2}$-boundary, we shall consider the $\Phi$-Laplace SPDE \eqref{eq:main_Ito_SPDE} such that $B$ satisfies (B1).

Let us first record a lemma for the situation that $d=1$. Let $\Ocal=(0,L)\subset\R$,
for some $L>0$. Set $I:=\overline{\Ocal}=[0,L]$. Let $\Delta=\partial_{xx}$
be the Dirichlet Laplace on $\Ocal$.

\begin{lem}
\label{lem:decay} Let $u,v\in V=H_{0}^{1}(I)$. For the line segment
$\gamma:[0,1]\to L^{2}(0,L),\ \lambda\mapsto\gamma(\lambda,u,v)$,
where $\gamma(\lambda,u,v)(x):=\partial_{x}u(x)+\lambda(\partial_{x}v(x)-\partial_{x}u(x))$,
we have that
\begin{equation}
_{H_{\phantom{0}}^{-1}}\langle A(u)-A(v),u-v\rangle_{H_{0}^{1}}\ge\frac{\lVert u-v\rVert_{H}^{2}}{L}\int_{0}^{1}\left(\int_{0}^{L}\left(\psi'(\gamma(\lambda,u,v)(x))\right)^{-1}\,dx\right)^{-1}\,d\lambda.\label{eq:decay}
\end{equation}
\end{lem}
\begin{proof}
First note that (C4) guarantees that the expression on the RHS of
\eqref{eq:decay} is almost surely less or equal zero as for $d\lambda$-a.e.
$\lambda\in[0,1]$, 
\[
\int_{0}^{L}\left(\psi'(\gamma(\lambda,u,v)(x))\right)^{-1}\,dx\leq2^{(s-1)\vee0}\int_{0}^{L}\left(a+b|\partial_{x}u(x)|^{s}+b|\partial_{x}v(x)|^{s}\right)\,dx<+\infty.
\]
We find the identity
\begin{align*}
 & _{H_{\phantom{0}}^{-1}}\langle A(u)-A(v),u-v\rangle_{H_{0}^{1}}=\int_{0}^{L}\left(\phi(\partial_{x}u(x))-\phi(\partial_{x}v(x))\right)(\partial_{x}u(x)-\partial_{x}v(x))\,dx\\
= & -\int_{0}^{L}\int_{0}^{1}\frac{d}{d\lambda}\phi(\gamma(\lambda,u,v)(x))(\partial_{x}u(x)-\partial_{x}v(x))\,d\lambda\,dx\\
= & \int_{0}^{L}\int_{0}^{1}\psi'(\gamma(\lambda,u,v)(x))(\partial_{x}u(x)-\partial_{x}v(x))^{2}\,d\lambda\,dx
\end{align*}
Let $\tilde{u},\tilde{v}$ denote continuous representatives of $u,v\in H_{0}^{1}(0,L)$.
Note that $\tilde{u}(0)=\tilde{v}(0)=0$. Using the fundamental theorem
of calculus and H\"{o}lder inequality yields 
\begin{align*}
 & \left|\tilde{u}(x)-\tilde{v}(x)\right|^{2}=\left(\int_{0}^{x}(\partial_{y}u(y)-\partial_{y}v(y))\,dy\right)^{2}\\
\leq & \left(\int_{0}^{L}\psi'(\gamma(\lambda,u,v)(x))(\partial_{x}u(x)-\partial_{x}v(x))^{2}\,dx\right)\times\left(\int_{0}^{L}\left(\psi'(\gamma(\lambda,u,v)(x))\right)^{-1}\,dx\right).
\end{align*}
Hence,
\[
\begin{aligned} & \frac{1}{L}\int_{0}^{L}\left|u(x)-v(x)\right|\,dx\\
\leq & \left(\int_{0}^{L}\psi'(\gamma(\lambda,u,v)(x))(\partial_{x}u(x)-\partial_{x}v(x))^{2}\,dx\right)\times\left(\int_{0}^{L}\left(\psi'(\gamma(\lambda,u,v)(x))\right)^{-1}\,dx\right).
\end{aligned}
\]
By dividing and integrating with respect to $\lambda$, we achieve
\[
\begin{gathered}\frac{\lVert u-v\rVert_{H}^{2}}{L}\int_{0}^{1}\left(\int_{0}^{L}\left(\psi'(\gamma(\lambda,u,v)(x))\right)^{-1}\,dx\right)^{-1}\,d\lambda\\
\le\int_{0}^{1}\int_{0}^{L}\psi'(\gamma(\lambda,u,v)(x))(\partial_{x}u(x)-\partial_{x}v(x))^{2}\,dx\,d\lambda.
\end{gathered}
\]
\end{proof}

Now, consider the general case that $d\ge1$. Let $\Ocal\subset\R^{d}$ be
a bounded, open and convex domain with piecewise $C^{2}$-boundary.
We denote $|\Ocal|:=\int_{\Ocal}\,dx$.
We define $\Psi:\R^{d}\to\R$
by 
\begin{equation}
\Psi(x):=\int_{0}^{\abs{x}}\psi(r)\,dr.\label{eq:convex-potential}
\end{equation}
We see that $\Psi$ is a radially symmetric convex function of at
most quadratic growth, and its first and second derivatives are given
by 
\[
D\Psi(x)=\phi(x)=\begin{cases}
0, & \text{if }x=0,\\
\psi(\abs{x})\frac{x}{\abs{x}}, & \text{if }x\ne0,
\end{cases}
\]
and by
$$ 
(D^{2}\Psi(x))_{(i,j)} 
= \frac{\psi(\abs{x})}{\abs{x}}\delta_{ij} + \left(\psi'(\abs{x})-\frac{\psi(\abs{x})}{\abs{x}}\right)\frac{x_{i}x_{j}}{\abs{x}^{2}},
$$
with special case $(D^2\Psi(x))_{(1,1)}=\psi'(|x|)$ for $d=1$.

Using the Gram matrix, one obtains that the eigenvalues of $D^{2}\Psi(x)$
are $\psi'(\abs{x})$ and $\frac{\psi(\abs{x})}{\abs{x}}$ with respective
multiplicities $1$ and $d-1$. In particular, 
\begin{equation}
\label{ellipticity}  
\langle D^{2}\Psi(x) h,h\rangle \ge \psi'(\abs{x})\wedge \left( \frac{\psi(\abs{x})}{\abs{x}}\right) |h|^2 
\quad\text{for every}\quad h\in\mathbb{R}^d . 
\end{equation} 
As above, for $u,v\in V,$ the operator
$A(u)$ can be defined via 
\[
\dualdel{V}{A(u)}{v}=\int_{\mathcal{O}}\langle\phi(\nabla u),\nabla v\rangle\,dx.
\]
For $u\in C^2(\mathcal O)$ we have the representation  
\begin{equation} 
\label{representation:A}
A(u) = - 
 \frac{\psi(\abs{\nabla u})}{\abs{\nabla u}} \Delta u -\left( \psi'(\abs{\nabla u})-\frac{\psi(\abs{\nabla u})}{\abs{\nabla u}}\right)
\sum_{i,j=1}^d \frac{\partial_{x_i}\partial_{x_j} u \partial_{x_i} u \partial_{x_j} u}{|\nabla u|^2}.  
\end{equation}
Also, on a formal level, 
\[
\begin{aligned}\langle A(u),v\rangle_{V}= & -\sum_{i=1}^{d}\int_{\mathcal{\Ocal}}\partial_{x_{i}}(\div(\phi(\nabla u))\,\partial_{x_{i}}v\,dx\\
= & \int_{\mathcal{\Ocal}}\div(\phi(\nabla u))\,\Delta v\,dx\\
= & \sum_{i=1}^{d}\int_{\mathcal{O}}\sprod{D^{2}\Psi(\nabla u)\nabla\partial_{x_{i}}u}{\nabla\partial_{x_{i}}v}\,dx.
\end{aligned}
\]

\begin{lem}
\label{lemma:stability} Let $d\ge2$. For $u,v\in H_{0}^{1}(\Ocal)$,
define $\gamma:[0,1]\to L^{2}(\Ocal;\R^{d})$, $\lambda\mapsto\gamma(\lambda,u,v)$
via $\gamma(\lambda,u,v)(x)=\nabla u(x)+\lambda(\nabla v(x)-\nabla u(x))$.
Set $\Lambda_{\min}(x):=\psi'(\abs{x})\wedge\frac{\psi(\abs{x})}{\abs{x}}$.
Then there exists $C=C(\Ocal,d)>0$,
\[
_{H_{\phantom{0}}^{-1}}\langle A(u)-A(v),u-v\rangle_{H_{0}^{1}}\ge C\norm{u-v}_{L^{2}}^{2}\int_{0}^{1}\left(\int_{\mathcal{O}}\left(\Lambda_{\min}(\gamma(\lambda,u,v)(x))\right)^{-\frac{d}{2}}\,dx\right)^{-\frac{2}{d}}\,d\lambda.
\]
\end{lem}
\begin{proof}
As in the assertion, $\Lambda_{\min}(x):=\psi'(\abs{x})\wedge\frac{\psi(\abs{x})}{\abs{x}}$. First note that (C4) and (C6) guarantee that the expression
on the RHS is well-defined as for $d\lambda$-a.e. $\lambda\in[0,1]$,
\[
\int_{\mathcal{O}}\left(\Lambda_{\min}(\gamma(\lambda,u,v)(x))\right)^{-\frac{d}{2}}\,dx\leq\int_{\mathcal{O}}\left(a+b|\nabla u(x)|^{\frac{ds}{2}}+b|\nabla v(x)|^{\frac{ds}{2}}\right)\,dx<+\infty.
\]
We find the identity
\begin{align*}
 & _{H_{\phantom{0}}^{-1}}\langle A(u)-A(v),u-v\rangle_{H_{0}^{1}}=\int_{\mathcal{O}}\sprod{\psi(\nabla u(x))-\psi(\nabla v(x))}{\nabla u(x)-\nabla v(x)}\,dx\\
= & -\int_{\mathcal{O}}\int_{0}^{1}\frac{d}{d\lambda}\psi(\gamma(\lambda,u,v)(x))(\nabla u(x)-\nabla v(x))\,d\lambda\,dx\\
= & \int_{\mathcal{O}}\int_{0}^{1}\sprod{D^{2}\Psi(\gamma(\lambda,u,v)(x))(\nabla u(x)-\nabla v(x))}{\nabla u(x)-\nabla v(x)}\,d\lambda\,dx,
\end{align*}
where $\Psi$ is as in \eqref{eq:convex-potential}. From the embedding
of $W_{0}^{1,\frac{2d}{d+2}}(\Ocal)$ into $L^{2}(\Ocal)$ and H\"{o}lder's
inequality, we derive 
\begin{align*}
 & \norm{u-v}_{L^{2}}\leq C\norm{u-v}_{W_{0}^{1,\frac{2d}{d+2}}}\\
\leq & C\left(\int_{\mathcal{O}}\sprod{D^{2}\Psi(\gamma(\lambda,u,v)(x))(\nabla u(x)-\nabla v(x))}{\nabla u(x)-\nabla v(x)}^{\frac{d}{d+2}}\Lambda_{\min}(\gamma(\lambda,u,v)(x))^{-\frac{d}{d+2}}\,dx\right)^{\frac{d+2}{2d}}\\
\leq & C\left(\int_{\mathcal{O}}\sprod{D^{2}\Psi(\gamma(\lambda,u,v)(x))(\nabla u(x)-\nabla v(x))}{\nabla u(x)-\nabla v(x)}\,dx\right)^{\frac{1}{2}}\\
 & \times\left(\int_{\mathcal{O}}\Lambda_{\min}(\gamma(\lambda,u,v)(x))^{-\frac{d}{2}}\,dx\right)^{\frac{1}{d}},
\end{align*}
where $C>0$ may change from line to line. Integration with respect
to $\lambda$ yields
\begin{align*}
 & \norm{u-v}_{L^{2}}^{2}\int_{0}^{1}\left(\int_{\mathcal{O}}\Lambda_{\min}(\gamma(\lambda,u,v)(x))^{-\frac{d}{2}}\,dx\right)^{-\frac{2}{d}}\,d\lambda\\
\le & C\int_{0}^{1}\int_{\mathcal{O}}\sprod{D^{2}\Psi(\gamma(\lambda))(\nabla u(x)-\nabla v(x))}{\nabla u(x)-\nabla v(x)}\,dx\,d\lambda
\end{align*}
Rearranging terms yields the result.
\end{proof}

\begin{prop}
\label{thm:decay} Suppose 
that $d\ge1$ with $\Ocal\subset\R^{d}$ bounded, open and convex. Let $u$ and $v$ be solutions to
\eqref{eq:main-spde} with $\phi$ as in \eqref{eq:phi_form} and
let $u_{0},v_{0}\in V=H_{0}^{1}(\Ocal)$. Let $d_{0}:=1\vee\frac{d}{2}$.
For $0\leq\alpha\leq1$, $l\ge1$. and $s\in(0,2]$ as in (C4), we
have that 
\begin{align*}
\lVert u_{t}-v_{t}\rVert_{H}^{2\alpha}\leq t^{-\alpha l}C(l,\alpha,\Ocal,d)\left(\frac{1}{t}\int_{0}^{t}(a|\Ocal|)^{\alpha d_{0}}+b^{\alpha d_{0}}\lVert u_{\tau}\rVert_{V}^{\alpha d_{0}s}+b^{\alpha d_{0}}\lVert v_{\tau}\rVert_{V}^{\alpha d_{0}s}\,d\tau\right)^{\frac{l}{d_{0}}}\lVert u_{0}-v_{0}\rVert_{H}^{2\alpha}.
\end{align*}
\end{prop}
\begin{proof}
Starting from Lemmas \ref{lem:decay} and \ref{lemma:stability},
where $\gamma(\lambda,u,v)(x)=\nabla u(x)+\lambda(\nabla v(x)-\nabla u(x))$,
\[
\frac{1}{2}\frac{d}{dt}\lVert u_{t}-v_{t}\rVert_{H}^{2}\leq-C\lVert u_{t}-v_{t}\rVert_{H}^{2}\int_{0}^{1}\left(\int_{\Ocal}\left(\psi'(\gamma(\lambda,u_{t},v_{t})(x))\right)^{-d_{0}}\,dx\right)^{-\frac{1}{d_{0}}}\,d\lambda,
\]
where in the $d=1$ case, $C=L^{-1}$. Gronwall's Lemma yields 
\[
\lVert u_{t}-v_{t}\rVert_{H}^{2}\leq\lVert u_{0}-v_{0}\rVert_{H}^{2}\exp\left(-2C\int_{0}^{t}\int_{0}^{1}\left(\int_{\Ocal}\left(\psi'(\gamma(\lambda,u_{\tau},v_{\tau})(x))\right)^{-d_{0}}\,dx\right)^{-\frac{1}{d_{0}}}\,d\lambda\,d\tau\right).
\]
For $l\geq1$, $x\geq0$, the elementary estimate $\log(\frac{x}{l})\le\frac{x}{l}-1$
implies 
\[
e^{x}\geq c_{l}x^{l}
\]
with $c_{l}:=\left(\frac{e}{l}\right)^{l}$. Applying Jensen's inequality
thrice, we get by (C4) that
\begin{align*}
 & \lVert u_{0}-v_{0}\rVert_{H}^{2\alpha}\geq c_{l}^{\alpha}\left(\int_{0}^{t}2C\int_{0}^{1}\left(\int_{\Ocal}\left(\psi'(\gamma(\lambda,u_{\tau},v_{\tau})(x))\right)^{-d_{0}}\,dx\right)^{-\frac{1}{d_{0}}}\,d\lambda\,d\tau\right)^{\alpha l}\lVert u_{t}-v_{t}\rVert_{H}^{2\alpha}\\
\geq & \ c_{l}^{\alpha}t^{\alpha l}\left(\frac{1}{t}\int_{0}^{t}\left(2C\int_{0}^{1}\left(\int_{\Ocal}\left(\psi'(\gamma(\lambda,u_{\tau},v_{\tau})(x))\right)^{-d_{0}}\,dx\right)^{-\frac{1}{d_{0}}}\,d\lambda\right)^{\alpha}\,d\tau\right)^{l}\lVert u_{t}-v_{t}\rVert_{H}^{2\alpha}\\
= & \ c_{l}^{\alpha}(2C)^{\alpha l}t^{\alpha l}\left(\frac{1}{t}\int_{0}^{t}\left(\int_{0}^{1}\left(\int_{\Ocal}\left(\psi'(\gamma(\lambda,u_{\tau},v_{\tau})(x))\right)^{-d_{0}}\,dx\right)^{-\frac{1}{d_{0}}}\,d\lambda\right)^{\alpha}\,d\tau\right)^{l}\lVert u_{t}-v_{t}\rVert_{H}^{2\alpha}.\\
\geq & c_{l}^{\alpha}(2C)^{\alpha l}t^{\alpha l}\left(\frac{1}{t}\int_{0}^{t}\left(\int_{0}^{1}\int_{\Ocal}\left(\psi'(\gamma(\lambda,u_{\tau},v_{\tau})(x))\right)^{-d_{0}}dx\,d\lambda\right)^{\alpha}\,d\tau\right)^{-\frac{l}{d_{0}}}\lVert u_{t}-v_{t}\rVert_{H}^{2\alpha}\\
\geq & c_{l}^{\alpha}(2C)^{\alpha l}t^{\alpha l}\frac{\lVert u_{t}-v_{t}\rVert_{H}^{2\alpha}}{\left(\frac{1}{t}\int_{0}^{t}\left((a|\Ocal|)^{\alpha d_{0}}+b^{\alpha d_{0}}\lVert u_{\tau}\rVert_{V}^{\alpha d_{0}s}+b^{\alpha d_{0}}\lVert v_{\tau}\rVert_{V}^{\alpha d_{0}s}\right)\,d\tau\right)^{\frac{l}{d_{0}}}}.
\end{align*}
\end{proof}

\section{\label{sec:Moment-estimates}Moment estimates}

Consider equation \eqref{eq:main_Ito_SPDE} with the same assumptions as stated in the beginning of Section \ref{sec:Stability}, in other words, we assume (C1)--(C6) and (B1).

Let us start with recording a second order functional inequality for the one-dimensional case that we will utilize later.
Let $d=1$. Let $\Ocal=(0,L)\subset\R$, for some $L>0$. Set $I:=\overline{\Ocal}=[0,L]$.
Let $\Delta=\partial_{xx}$ be the Dirichlet Laplace on $I$. Let
us equip $W_{0}^{1,1}(I)$ with the norm $\|u\|_{W_{0}^{1,1}(I)}=\int_{I}|\partial_{x}u|\,dx.$

\begin{lem}[compare with Lemma 2.2. in \cite{ESvRS}]
\label{lem:a-priori} For $u\in C_{0}^{\infty}(I)$ and $a\ge 0$, $b>0$ and $s\in(0,2]$
as in Assumption (C4), there exists constants $C(L,a,s)>0$ and $C(L,b,s)>0$
such that
\[
\left(\int_{I}|\partial_{xx}u|\,dx\right)^{s^{\ast}}\le s^{\ast}\int_{I}\psi^{\prime}(\partial_{x}u)|\partial_{xx}u|^{2}\,dx+C(L,b,s)\|u\|_{W_{0}^{1,1}(I)}^{1\vee(2-s)}+C(L,a,s),
\]
where
\begin{equation}
s^{\ast}:=(2-s)\vee\frac{4-s}{2+s}.\label{eq:s^ast}
\end{equation}
\end{lem}

\begin{proof}
Let $u\in C_{0}^{\infty}(I)$. Let $\alpha\in[\frac{1}{2},2-\frac{s}{2})$.
By H\"{o}lder inequality, Young inequality and (C4),
\[
\begin{aligned} & \left(\int_{I}|\partial_{xx}u|\,dx\right)^{\alpha}\\
\le & \left(\int_{I}\psi'(\partial_{x}u)|\partial_{xx}u|^{2}\,dx\right)^{\frac{\alpha}{2}}\left(\int_{I}\frac{1}{\psi'(\partial_{x}u)}\,dx\right)^{\frac{\alpha}{2}}\\
\le & \frac{\alpha}{2}\int_{I}\psi'(\partial_{x}u)|\partial_{xx}u|^{2}\,dx+\frac{2-\alpha}{2}\left(\int_{I}\left(a+b|\partial_{x}u|^{s}\right)\,dx\right)^{\frac{\alpha}{2-\alpha}}\\
\le & \frac{\alpha}{2}\int_{I}\psi'(\partial_{x}u)|\partial_{xx}u|^{2}\,dx+2^{\frac{2\alpha-2}{2-\alpha}\vee0}\left[\frac{2-\alpha}{2}(aL)^{\frac{\alpha}{2-\alpha}}+b^{\frac{\alpha}{2-\alpha}}\frac{2-\alpha}{2}\left(\int_{I}|\partial_{x}u|^{s}\,dx\right)^{\frac{\alpha}{2-\alpha}}\right].
\end{aligned}
\]
Furthermore, by Jensen's inequality and the embedding $W_{0}^{1,1}(I)\hookrightarrow C(I)$
(note that $u(0)=0$),
\[
\begin{aligned} & L^{\frac{2-s}{2}}\left(\int_{I}|\partial_{x}u|^{s}\,dx\right)^{\frac{2}{s}}\\
\le & \int_{I}|\partial_{x}u|^{2}\,dx=-\int_{I}u\partial_{xx}u\,dx\\
\le & \|\partial_{x}u\|_{L^{1}(I)}\int_{I}|\partial_{xx}u|\,dx.
\end{aligned}
\le\|u\|_{\infty}\int_{I}|\partial_{xx}u|\,dx
\]
Since $\alpha<2-\frac{s}{2}$, we can apply Young's inequality again
and get altogether,
\[
\begin{aligned} & 2\left(\int_{I}|\partial_{xx}u|\,dx\right)^{\alpha}\\
\le & \alpha\int_{I}\psi'(\partial_{x}u)|\partial_{xx}u|^{2}\,dx+C(L,b,\alpha,s)\left(\int_{I}|\partial_{x}u|^{s}\,dx\right)^{\frac{\alpha}{2-\alpha}}+C(L,a,\alpha)\\
\le & \alpha\int_{I}\psi'(\partial_{x}u)|\partial_{xx}u|^{2}\,dx+C(L,b,\alpha,s)\left(\int_{I}|\partial_{xx}u|\,dx\right)^{\frac{\alpha s}{4-2\alpha}}\|\partial_{x}u\|_{L^{1}(I)}^{\frac{\alpha s}{4-2\alpha}}+C(L,a,\alpha)\\
\le & \alpha\int_{I}\psi'(\partial_{x}u)|\partial_{xx}u|^{2}\,dx+\left(\int_{I}|\partial_{xx}u|\,dx\right)^{\alpha}+C(L,b,\alpha,s)\|\partial_{x}u\|_{L^{1}(I)}^{\frac{\alpha s}{4-2\alpha-s}}+C(L,a,\alpha).
\end{aligned}
\]
Choosing $\alpha=s^{\ast}=(2-s)\vee\frac{4-s}{2+s}$ yields the proof.
\end{proof}

Now, consider the general multivariate case $d\ge 1$. Let $\Ocal\subset\R^{d}$ be a bounded, open and convex domain with piecewise
$C^{2}$-boundary. To achieve the main result, we first show a pathwise regularity result
for \eqref{eq:main_Ito_SPDE}.
This result combined with moment estimates for the finite dimensional
Galerkin approximations of \eqref{eq:main_Ito_SPDE} yields the existence
of a unique invariant measure for \eqref{eq:main_Ito_SPDE}. 

\begin{lem}\label{lem_moments} For every $0<s\le2$, and every $1<\alpha\leq d\frac{2-s}{d-s}$,
there exists $C=C(d,s,\Ocal)>0$ such that for every $u\in C_{0}^{\infty}(\mathcal{O})$, 
\[
\left(\int_{\mathcal{O}}\sum_{i=1}^{d}\sum_{i=1}^{d}\lvert\partial_{x_{i}}\partial_{x_{j}}u\rvert^{\alpha}\,dx\right)^{\frac{s^{\ast}}{\alpha}}\leq C\left(1+\sum_{i=1}^{d}\int_{\mathcal{O}}\sprod{D^{2}\Psi(\nabla u)\nabla\partial_{x_{i}}u}{\nabla\partial_{x_{i}}u}\,dx\right),
\]
where $s^{\ast}$ is as in \eqref{eq:s^ast} and $\Psi$ is defined as in \eqref{eq:convex-potential}.
\end{lem} \begin{proof} Set $\Lambda_{\min}(x):=\psi'(\abs{x})\wedge\frac{\psi(\abs{x})}{\abs{x}}$.
Let $1<\alpha\leq d\frac{2-s}{d-s}$. First, by using H\"{o}lder's and Young's inequalities,
we get for $\varepsilon>0$,
\begin{align*}
 & \left(\int_{\mathcal{O}}\sum_{i=1}^{d}\sum_{j=1}^{d}\lvert\partial_{x_{i}}\partial_{x_{j}}u\rvert^{\alpha}dx\right)^{\frac{2-s}{\alpha}}\ \\
\le & C\left(\int_{\mathcal{O}}\sum_{i=1}^{d}\sprod{D^{2}\Psi(\nabla u)\nabla\partial_{x_{i}}u}{\nabla\partial_{x_{i}}u}^{\frac{\alpha}{2}}\cdot\Lambda_{\min}^{-\frac{\alpha}{2}}(\nabla u)\,dx\right)^{\frac{2-s}{\alpha}}\\
\le & C\left(\int_{\mathcal{O}}\sum_{i=1}^{d}\sprod{D^{2}\Psi(\nabla u)\nabla\partial_{x_{i}}u}{\nabla\partial_{x_{i}}u}\,dx\right)^{\frac{2-s}{2}}\times\left(\int_{\mathcal{O}}\Lambda_{\min}^{-\frac{\alpha}{2-\alpha}}(\nabla u)\,dx\right)^{\frac{(2-s)(2-\alpha)}{2\alpha}}\\
\le & C\left[\frac{2-s}{2}\varepsilon^{-\frac{2}{2-s}}\int_{\mathcal{O}}\sum_{i=1}^{d}\sprod{D^{2}\Psi(\nabla u)\nabla\partial_{x_{i}}u}{\nabla\partial_{x_{i}}u}\,dx+\frac{s}{2}\varepsilon^{\frac{2}{s}}\left(\int_{\mathcal{O}}\Lambda_{\min}^{-\frac{\alpha}{2-\alpha}}(\nabla u)\,dx\right)^{\frac{(2-s)(2-\alpha)}{s\alpha}}\right].
\end{align*}
We note that $W_{0}^{2,\alpha}(\Ocal)$ is continuously embedded into
$W_{0}^{1,\frac{s\alpha}{2-\alpha}}(\Ocal)$ for $1<\alpha\leq d\frac{2-s}{d-s}$. This embedding, together
with (C4), yields 
\[
\begin{aligned}\left(\int_{\mathcal{O}}\Lambda_{\min}^{-\frac{\alpha}{2-\alpha}}(\nabla u)\,dx\right)^{\frac{(2-s)(2-\alpha)}{s\alpha}}\le C\left(1+\int_{\mathcal{O}}|\nabla u|^{\frac{s\alpha}{2-\alpha}}\,dx\right)^{\frac{(2-s)(2-\alpha)}{s\alpha}}\le C\left(1+\norm{u}_{W_{0}^{2,d\frac{2-s}{d-s}}(\Ocal)}^{s^{\ast}}\right).\end{aligned}
\]
Choosing $\varepsilon$ small enough completes the proof. \end{proof}

We shall prove the existence of the unique invariant measure for the ergodic semigroup and the main concentration result now.

\begin{thm}\label{support} There exists a unique invariant Borel probability
measure $\mu$ for the semigroup $\{P_t\}$ of equation \eqref{eq:main_Ito_SPDE} that is concentrated on the subset $W^{2,\alpha}_0(\Ocal)\cap W^{1,2}_0(\Ocal)\subset L^2(\Ocal)$ for any $1<\alpha\leq d\frac{2-s}{d-s}$ if $d\ge 2$ and $\alpha=1$ if $d=1$ with
\[
\int_{H}\|u\|_{W_{0}^{2,\alpha}(\Ocal)}^{s^{\ast}}\,\mu(du)+\int_{H}\lVert u\rVert_{V}^{s^{\ast}}\,\mu(du)+\sum_{i=1}^d \int_{H} \langle D^2 \Psi \left( \nabla u\right) \partial_{x_i} 
\nabla u , \partial_{x_i}\nabla u \rangle_H\,\mu(du) <+\infty,
\]
where $s^{\ast}$ is as in \eqref{eq:s^ast} and $s\in(0,2]$ is as
in assumption (C4) and $\Psi$ is defined as in \eqref{eq:convex-potential}.\end{thm} 

\begin{proof}
For a complete orthonormal system $\{e_{k}\}_{k\in\N}$ of $H=L^2(\Ocal)$,
and for $H_{m}:=\operatorname{span}\{e_{k}\;\colon\;1\leq k\leq m\}$,
we consider the (finite dimensional) approximating stochastic differential equation in $H$ 
\begin{equation}\label{philaplacian_approx}\begin{split}
du^{m}_t+Au^{m}_t\,dt & =P_{m}B\,dW_{t}\\
u^{m}_0 & =P_{m}u_{0},
\end{split}\end{equation}
where $A$ is as in \eqref{eq:A_defi}.
Note that here, we cannot use the standard Galerkin approximation
as in \cite{ESvR}, for there is no guarantee that the nonlinear operator
$A$ maps $H_{m}$ into $V=H_{0}^{1}(\Ocal)$. One option would be to
add an extra smoothing step for the operator $A$ as e.g. in \cite[Proof of Theorem 3.1]{GessToelle15}.
Instead, we follow the strategy from \cite[Section 6.3.1]{GessToelle14}
to find solutions $u^{m}$ converging to $u$ in $L^{2}(\Omega,C([0,T];H))$
with the aim to consider the corresponding invariant measures $\nu_{m}$.
For $n\in\N$, let $J_{n}:=\left(\operatorname{Id}-\frac{\Delta}{n}\right)^{-1}$
and $T_{n}:=-\Delta J_{n}=n(\operatorname{Id}-J_{n})$, respectively,
be the resolvent and the smooth Yosida approximation of the nonnegative definite
differential operator $-\Delta=-\sum_{i=1}^{d}\partial_{x_{i}}\partial_{x_{i}}$,
respectively. Let $\langle u,v\rangle_{n}:=\langle u,T_{n}v\rangle_{H}$.
The induced
norms $\lVert\cdot\rVert_{n}$ are equivalent to the $H$-norm and
converge monotonically as follows
\[
\lVert u\rVert_{n}\underset{n\to+\infty}{\uparrow}\begin{cases}
\lVert u\rVert_{V}, & u\in V,\\
+\infty, & \text{otherwise.}
\end{cases}
\]
Applying It\^{o}'s formula to $\lVert u^{m}(t)\rVert_{n}^{2}$, we
get 
\begin{equation}
\lVert u_{t}^{m}\rVert_{n}^{2}=\lVert u_{0}^{m}\rVert_{n}^{2}-2\int_{0}^{t}\langle A(u_{\tau}^{m}),T_{n}(u_{\tau}^{m})\rangle\,d\tau+\int_{0}^{t}\lVert P_{m}B\rVert_{L_{2}(U,H_{m})}^{2}\,d\tau+M_{t}^{n}\label{eq:ito-1}
\end{equation}
where $t\mapsto M_{t}^{n}=2\int_{0}^{t}\langle u_{\tau}^{m},P_{m}B\,dW_{\tau}\rangle_{n}$
is a local martingale. From \cite[Proposition 8.2]{VBMR} and (C3),
applied to \cite[Proposition 7.1]{GessToelle14}, it follows that
\[
\langle A(u_{\tau}^{m}),T_{n}(u_{\tau}^{m})\rangle_{H}\ge0\quad d\tau\text{-a.e}.
\]
We note that the convexity assumption of for the domain $\Ocal$ enters here, as it is needed to apply \cite[Proposition 8.2]{VBMR}.
By taking expectation, together with Fatou's lemma, employing the
bounds $\lVert u_{0}^{m}\rVert_{n}\leq\lVert u_{0}\rVert_{V}$ and
$\lVert P_{m}B\rVert_{L_{2}(U,H^{m})}^{2}\leq\lVert B\rVert_{L_{2}(U,V)}^{2}$,
it follows from \eqref{eq:ito-1} for every $t\ge1$ that
\begin{equation} 
\label{est:Thm3.4}
\frac{1}{t}\E\left(\int_{0}^{t}\sum_{i=1}^{d}\int_{\Ocal}\langle D^{2}\Psi(\nabla u_{\tau}^{m})\nabla\partial_{x_{i}}u_{\tau}^{m},\nabla\partial_{x_{i}}u_{\tau}^{m}\rangle\,dx\,d\tau\right)\leq\frac{1}{2}\left(\E\lVert u_{0}\rVert_{V}^2+\lVert B\rVert_{L_{2}(U,V)}^{2}\right)=:K_{1},
\end{equation} 
independently of $m$. On the other hand, the application of It\^{o}'s
formula to $t\mapsto\lVert u_{t}^{m}\rVert_{H}^{2}$ yields 
\[
\lVert u_{t}^{m}\rVert_{H}^{2}=\lVert u_{0}^{m}\rVert_{H}^{2}-2\int_{0}^{t}\langle A(u_{\tau}^{m}),u_{\tau}^{m}\rangle_{H}\,d\tau+\int_{0}^{t}\lVert P_{m}B\rVert_{L_{2}(U,H_{m})}^{2}\,d\tau+N_{t}^{m}
\]
where $t\mapsto N_{t}^{m}=2\int_{0}^{t}\langle u_{\tau}^{m},P_{m}B\,dW_{\tau}\rangle_{H}$
and $\langle A(u^{m}),u^{m}\rangle=\int_{\Ocal}\langle\phi(\nabla u^{m}),\nabla u^{m}\rangle\,dx$.
Note that by (C5) and Lemma \ref{lem:growth}, we get
\[
\psi(r)\cdot r\geq c\lvert r\rvert^{1\vee(2-s)}-K,
\]
for suitable constants $c>0$ and $K\geq0$. This combined with the
same arguments as before entails that for every $t\geq1$ 
\begin{equation}\label{K2bound}
\frac{1}{t}\E\left(\int_{0}^{t}\int_{\Ocal}\lvert\nabla u_{\tau}^{m}\rvert^{1\vee(2-s)}\,dx\,d\tau\right)\leq\frac{1}{2c}\left(\lVert u_{0}\rVert_{H}^2+\lVert B\rVert_{L_{2}(U,H)}^{2}+K\right)=:K_{2}
\end{equation}
for a positive constant $K_{2}>0$.

Let us treat the case $d=1$ separately first, so let us assume that $\Ocal=(0,L)$.
By the above estimates, it is easy to prove that the family of probability measures $\left\{\frac{1}{t}\int_0^t P_\tau^\ast\delta_{u_0}\,d\tau\right\}_{t\ge 1}$ (where $\delta_{u_0}$ denotes the Dirac measure in $u_0$) is tight in the finite dimensional space $H_m$, and we thus find that by the Krylov-Bogoliubov existence theorem \cite[Theorem 3.1.1]{DaPrZergodicity} that
there exists an invariant measure $\nu_m$ for the Markov semigroup of \eqref{philaplacian_approx} such that for arbitrary $N>0$ 
\begin{equation}
\int_{H}\int_{I}\left(\lvert\partial_{x}u\rvert^{1\vee(2-s)}dx\wedge N\right)\nu_{m}(du)+\int_{H}\int_{I}\left(\psi'(\partial_{x}u)|\partial_{xx}u|^{2}dx\wedge N\right)\nu_{m}(du)\leq K_{3},\label{eq:moment-ito}
\end{equation}
where $K_{3}=K_{1}+K_{2}$. An application of the monotone convergence
lemma yields
\begin{equation}
\sup_{m\geq1}\int_{H}\int_{I}\lvert\partial_{x}u\rvert^{1\vee(2-s)}\,dx\,\nu_{m}(du)+\sup_{m\geq1}\int_{H}\int_{I}\psi'(\partial_{x}u)|\partial_{xx}u|^{2}\,dx\,\nu_{m}(du)\leq K_{3}.\label{eq_fatou}
\end{equation}
Combining \eqref{eq:moment-ito} with Lemma \ref{lem:a-priori}, we
obtain 
\begin{equation}
\sup_{m\ge1}\int_{H}\left(\int_{0}^{L}\lvert\partial_{xx}u\rvert\,dx\right)^{s^{\ast}}\,\nu_{m}(du)<+\infty.\label{eq:moment}
\end{equation}
Due to the compactness of the embedding $W_{0}^{2,1}(0,L)\subset W_{0}^{1,1}(0,L)$,
we get that the sequence of measures $\{\nu_{m}\}_{m\geq1}$ is tight
w.r.t. the $W^{1,1}$-topology. Let $\nu$ denote the limit of a weakly
convergent subsequence in the sense of weak convergence of probability measures. For a function $u\in L_{\loc}^{1}(B)$, we define its total variation on an open subset $B\subset\R^d$ by
$$[Du](B):=\sup\left\{ \int_{B}u\div v\,dx\;:\;v\in C^{1}(\overline{B},\R^{d}),\;\lVert v\rVert_{\infty}\le1\right\}.$$
Recall that $u\in L^{1}(B)$ belongs to the space $BV(B)$
of functions of bounded variation if $[Du](B)<+\infty$. 
Note that
$$u\mapsto\begin{cases} [D\partial_{x}u](\R),&\text{ if }\partial_x u\in BV(\R)\cap L^1(0,L),\;u\in W^{1,1}_0(0,L),\\ +\infty,&\text{ if } \partial_x u\in L^1(0,L)\setminus BV(\R),\;u\in W^{1,1}_0(0,L),\end{cases}$$
is the lower semi-continuous envelope of
$$u\mapsto\int_0^L|\partial_{xx} u|\,dx,\quad u\in W^{2,1}_0(0,L),$$ in $W^{1,1}_0(0,L)$.
The lower semi-continuity of $u\mapsto [D\partial_{x}u](\R)$ w.r.t the strong $W^{1,1}_0(0,L)$-topology, 
together with \eqref{eq:moment} yields by \cite[Proposition 1.62 (a)]{AFP20},
\[
\int_{H}\left([D\partial_{x}u](\R)\right)^{s^{\ast}}\,\nu(du)<+\infty.
\]
However, by the tightness, it can be seen by standard arguments that the support of $\nu$ is contained in $W_0^{2,1}(0,L)$ and thus
\[
\int_{H}\left(\int_0^L\partial_{xx}u\,dx\right)^{s^{\ast}}\,\nu(du)<+\infty.
\]
From the boundedness of the embedding of $W_{0}^{2,1}(0,L)$ into
$W_{0}^{1,2}(0,L)$ it follows that 
\[
\int_{H}\lVert u\Vert_{V}^{s^{\ast}}\,\nu(du)<+\infty.
\]
Hence the support of $\nu$ is included in $W_0^{2,1}(0,L)\cap H_0^{1}(0,L)$.

Let us suppose $d\ge 1$ now. 
By the above estimates, it is easy to prove that the family of probability measures $\left\{\frac{1}{t}\int_0^t P_\tau^\ast\delta_{u_0}\,d\tau\right\}_{t\ge 1}$ is tight in the finite dimensional space $H_m$ and thus by the Krylov-Bogoliubov existence theorem \cite[Theorem 3.1.1]{DaPrZergodicity},
as the bound \eqref{K2bound} is independent of
$t$, there exists an invariant measure $\nu_{m}$ for \eqref{philaplacian_approx}.
To show uniqueness, note that for any Lipschitz function $F$ and
$\nu_{m}$-almost every $u_{0},v_{0}$ in $H_{m}$, Proposition \ref{thm:decay}
yields (for $\alpha=\frac{1}{2}$ and $l=d_{0}=1\vee\frac{d}{2})$
that
\begin{equation}\label{eq:equilibrium}
\begin{split}
 & \abs{\E\left(\frac{1}{t}\int_{0}^{t}F(u_{\tau})\,d\tau\right)-\E\left(\frac{1}{t}\int_{0}^{t}F(v_{\tau})\,d\tau\right)}\\
\leq & \operatorname{Lip}(F)\,\E\left(\frac{1}{t}\int_{0}^{t}\lVert u_{\tau}-v_{\tau}\rVert_{H}\,d\tau\right)\\
\le & t^{-\frac{d_{0}}{2}}C\left(1+\frac{1}{t}\E\left(\int_{0}^{t}\lVert u_{\tau}\rVert_{V}^{\frac{sd_{0}}{2}}+\lVert v_{\tau}\rVert_{V}^{\frac{sd_{0}}{2}}\,d\tau\right)\right)\lVert u_{0}-v_{0}\rVert_{H}\\
\le & t^{-\frac{d_{0}}{2}}C\left(1+\lVert u_{0}\rVert_{V}^{\frac{sd_{0}}{2}}+\lVert v_{0}\rVert_{V}^{\frac{sd_{0}}{2}}\right)\lVert u_{0}-v_{0}\rVert_{H}\longrightarrow0,
\end{split}\end{equation}
as $t\to+\infty,$ where $C>0$ may change from line to line. The latter implies uniqueness of the invariant measure $\nu_m$ of the ergodic semigroup of \eqref{philaplacian_approx} by standard arguments, see e.g. \cite[Section 2.1, proof of Theorem 1.2 (iii)]{LiuToelle}.

Now, it follows from ergodicity that for arbitrary
$L>0$ 
\begin{equation}
\int_{H}\sum_{i=1}^{d}\int_{\mathcal{O}}\left(\sprod{D^{2}\Psi(\nabla u)\nabla u_{x_{i}}}{\nabla u_{x_{i}}}dx\wedge L\right)\,\nu_{m}(du)\leq C.\label{moment}
\end{equation}
The monotone convergence lemma yields 
\begin{equation}
\sup_{m\geq1}\int_{H}\sum_{i=1}^{d}\int_{\mathcal{O}}\left(\sprod{D^{2}\Psi(\nabla u)\nabla u_{x_{i}}}{\nabla u_{x_{i}}}dx\right)\,\nu_{m}(du)\leq C.\label{eq_fatou-1}
\end{equation}
Combining \eqref{eq_fatou-1} with Lemma \ref{lem_moments}, we obtain
\begin{equation}
\sup_{m\ge1}\int_{H}\norm{u}_{W_{0}^{2,(2-s)\frac{d}{d-s}}}^{s^{\ast}}\,\nu_{m}(du)<+\infty.\label{eq:moment-1}
\end{equation}
Keeping assumption (C6) in mind, the compactness of the embedding $W_{0}^{2,d\frac{2-s}{d-s}}(\mathcal{O})$
into $V=W_{0}^{1,2}(\Ocal)$ for $s<\frac{4}{d}$ implies that
the sequence of measures $\{\nu_{m}\}_{m\geq1}$ is tight in $V$. Thus,
there exists an invariant measure $\mu$ such that 
\[
\int_{H}\lVert u\Vert_{V}^{s^{\ast}}\,\mu(du)<+\infty.
\]
The moment estimate for the $W_{0}^{2,\alpha}(\Ocal)$-norm, where $\alpha$ is as in the statement of the theorem, follows by lower semi-continuity
from \eqref{eq:moment-1} and \cite[Proposition 1.62 (a)]{AFP20}.

The uniqueness of $\mu$ can be proved by the same arguments as for $\nu_m$, see the computation \eqref{eq:equilibrium} combined with \cite[Section 2.1, proof of Theorem 1.2 (iii)]{LiuToelle}.
\end{proof}

As a consequence, we obtain the proof of Theorem \ref{thm:main_in_intro} (i).

\begin{cor}[Stochastic $p$-Laplace]
Suppose that $\psi(r)=|r|^{p-2}r$, $p\in (1,2)$. Then, under assumptions (C6) and (B1) we get that the unique invariant measure $\mu$ is concentrated on $W_0^{2,\frac{dp}{d+p-2}}(\Ocal)\cap H^1_0(\Ocal)\subset L^2(\Ocal)$ and satisfies
\[
\int_{L^2(\Ocal)}\|u\|_{W_{0}^{2,\frac{dp}{d+p-2}}(\Ocal)}^{p}\,\mu(du)+\int_{L^2(\Ocal)}\lVert u\rVert_{H^1_0(\Ocal)}^{p}\,\mu(du)+ \int_{L^2(\Ocal)}\int_\Ocal|\nabla u|^{p-2}|D^2 u|^2\, dx\,\mu(du) <+\infty.
\]
\end{cor}
\begin{proof}
Note that here, $s^\ast=p$.
The result follows from the fact that for $\psi(r)=|r|^{p-2}r$, $\Psi(x)=\frac{1}{p}|x|^p$, 
\[(D^2\Psi(x))_{(i,j)}=|x|^{p-2}\delta_{ij}-(2-p)|x|^{p-2}\frac{x_i x_j}{|x|^2}.\]
\end{proof}

\section{\label{sec:decayandexamples}Decay estimate and examples}

Proposition \ref{thm:decay} is applied to prove the following main decay estimate for
the semigroup $\{P_{t}\}$.
\begin{thm}
\label{thm:semigroup} Let $\{P_{t}\}$ be the semigroup from \eqref{eq:defi_semigroup}
and let $\mu$ denote its unique invariant measure. Let $d_{0}:=1\vee\frac{d}{2}$. Let $s^{\ast}=(2-s)\vee\frac{4-s}{2+s}$. Let
$0<\beta\leq\beta^{\star}:=1\wedge\frac{8-2s}{s(2+s)d_{0}}\in\left[\frac{1}{2d_{0}},1\right]$. Then, under the assumptions of Proposition \ref{thm:decay}, there exist non-negative constants $C_{1}$, $C_{2}$
and $C_{3}$ such that 
\[
\limsup_{t\to+\infty}\left[t^{\frac{s^{\ast}}{s}}\frac{\lvert P_{t}f(u)-P_{t}f(v)\rvert}{\lVert u-v\rVert_{H}^{\beta}}\right]\leq C_{1}|f|_{\beta}\left(C_{2}+C_{3}\int_{H}\lVert u\rVert_{V}^{s^{\ast}}\,\mu(du)\right),
\]
for every $u,v\in V$ and every $f:H\to\R$ that is bounded and $\beta$-H\"{o}lder-continuous, i.e.,
\[
\sup_{\substack{u,v\in H\\
u\not=v
}
}\frac{\lvert f(u)-f(v)\rvert}{\lVert u-v\rVert_{H}^{\beta}}=:\lvert f\rvert_{\beta}<+\infty.
\] 
\end{thm}

\begin{proof}
We use Proposition \ref{thm:decay} and get for $\frac{\beta}{2}=:\alpha\le\frac{1}{2}\wedge\frac{s^{\ast}}{sd_{0}}=\frac{\beta^{\star}}{2}$,
$l:=\frac{s^{\ast}}{\alpha d_{0}s}\ge1$, 
\begin{align*}
 & \lvert P_{t}f(u)-P_{t}f(v)\rvert=\lvert\E[f(u_{t})-f(v_{t})]\rvert\leq\lvert f\rvert_{2\alpha}\E\left[\lVert u_{t}-v_{t}\rVert_{H}^{2\alpha}\right]\\
 & \leq|f|_{2\alpha}t^{-\frac{s^{\ast}}{s}}C(l,\alpha,\Ocal,d)\E\left[\left(\frac{1}{t}\int_{0}^{t}\left((a|\Ocal|)^{\alpha d_{0}}+b^{\alpha d_{0}}\lVert u_{\tau}\rVert_{V}^{\alpha d_{0}s}+b^{\alpha d_{0}}\lVert v_{\tau}\rVert_{V}^{\alpha d_{0}s}\right)\,d\tau\right)^{\frac{l}{d_{0}}}\right]\lVert u_{0}-v_{0}\rVert_{H}^{2\alpha}\\
 & \leq|f|_{\beta}t^{-\frac{s^{\ast}}{s}}C(l,\alpha,\Ocal,d)3^{(\frac{l}{d_{0}}-1)\vee0}\E\left(\frac{1}{t}\int_{0}^{t}\left((a|\Ocal|)^{\frac{s^{\ast}}{s}}+b^{\frac{s^{\ast}}{s}}\lVert u_{\tau}\rVert_{V}^{s^{\ast}}+b^{\frac{s^{\ast}}{s}}\lVert v_{\tau}\rVert_{V}^{s^{\ast}}\right)\,d\tau\right)\lVert u_{0}-v_{0}\rVert_{H}^{\beta}.
\end{align*}
Note that $\E\left(\frac{1}{t}\int_{0}^{t}\lVert u_{\tau}\rVert_{V}^{s^{\ast}}\,d\tau\right)$
converges to $\int_{H}\lVert u\rVert_{V}^{s^{\ast}}\,\mu(du)$ due
to the ergodicity of $\{P_{t}\}$. 
\end{proof}

As a consequence, we obtain the proof of Theorem \ref{thm:main_in_intro} (ii).
As an application, we find the following
convergence rates for the main examples.

\subsection{Stochastic mean curvature flow}

For the stochastic mean curvature flow equation in $1+1$ dimensions,
i.e. 
\[
A(u)=-\frac{\partial_{xx}u}{1+(\partial_{x}u)^{2}}=-\frac{\partial}{\partial x}\arctan(\partial_{x}u),\quad\psi(r)=\arctan(r),
\]
we have $s=2$, $\beta^{\star}=\frac{1}{2}$, so our result recovers
the result in \cite{ESvRS}.
\begin{cor}
\label{cor1} For $f$ H\"{o}lder continuous with $\beta=\frac{1}{2}$,
we have that 
\[
\limsup_{t\to+\infty}\left[t^{\frac{1}{4}}\frac{\lvert P_{t}f(u)-P_{t}f(v)\rvert}{\lVert u-v\rVert_{L^{2}(0,L)}^{\beta}}\right]\leq C_{1}\lvert f\rvert_{\frac{1}{2}}\left(C_{2}+C_{3}\int_{H}\lVert u\rVert_{V}^{\frac{1}{2}}\,\mu(du)\right),
\]
where the constants do not depend on $f$.
\end{cor}

\subsection{Singular $p$-Laplace}

Let $d_{0}:=1\vee\frac{d}{2}$. Consider, for $p\in(1,2)$, the operator
\[
A(u)=-\div\left(|\nabla u|^{p-2}\nabla u\right),\quad\psi(r)=|r|^{p-2}r.
\]
With $s=2-p$, we get $\beta^{\star}=\frac{8-2(2-p)}{(2-p)(4-p)d_{0}}\wedge1$,
$s^{\ast}=p$ and therefore get the following result which improves
the result in \cite{LiuToelle}, see also Remark \ref{rem:improvement}.
\begin{cor}
\label{cor3} For $f$ H\"{o}lder continuous with $0<\beta\le\beta^{\star}$,
we have that
\[
\limsup_{t\to+\infty}\left[t^{\frac{p}{2-p}}\frac{\lvert P_{t}f(u)-P_{t}f(v)\rvert}{\lVert u-v\rVert_{L^{2}(\Ocal)}^{\beta}}\right]\leq C_{1}\lvert f\rvert_{\beta}\left(C_{2}+C_{3}\int_{H}\lVert u\rVert_{V}^{p}\,\mu(du)\right),
\]
where the constants do not depend on $f$.
\end{cor}

\subsection{Non-Newtonian Fluids}

Let $d_{0}:=1\vee\frac{d}{2}$. Consider the following operator for
$p\in(1,2)$, 
\[
A(u)=-\div\left(\left(1+|\nabla u|^{2}\right)^{\frac{p-2}{2}}\nabla u\right),\quad\psi(r)=\left(1+r^{2}\right)^{\frac{p-2}{2}}r.
\]
As before, we have $s=2-p$, and get $\beta^{\star}=\frac{8-2(2-p)}{(2-p)(4-p)d_{0}}\wedge1$,
$s^{\ast}=p$. Therefore, we obtain the following. 
\begin{cor}
\label{cor2} For $f$ H\"{o}lder continuous with $0<\beta\le\beta^{\star}$,
we have that
\[
\limsup_{t\to+\infty}\left[t^{\frac{p}{2-p}}\frac{\lvert P_{t}f(u)-P_{t}f(v)\rvert}{\lVert u-v\rVert_{L^{2}(\Ocal)}^{\beta}}\right]\leq C_{1}\lvert f\rvert_{\beta}\left(C_{2}+C_{3}\int_{H}\lVert u\rVert_{V}^{p}\,\mu(du)\right),
\]
where the constants do not depend on $f$.
\end{cor}

\section{\label{sec:Maximal-dissipativity}Maximal dissipativity of the associated
Kolmogorov operator}

Set
\[\kappa(r) = \psi^\prime (\abs{r}) \wedge  
\frac{\psi (\abs{r})}{\abs{r}},\]
i.e. $\Lambda_{\min}(x)=\kappa(|x|)$, where $\Lambda_{\min}$ is defined as in Lemma \ref{lemma:stability},
and set 
\[
D_{0}:=\left\{ u\in W_{0}^{1,1}(\mathcal{O})\;:\;  \sqrt{\kappa \left( |\nabla u|\right)}\, \partial_{x_i}
\partial_{x_j} u\in L^2 (\mathcal{O}) , 1\le i,j \le  d \right\} .
\]
Recall that Theorem \ref{support}  
yields that the support of the invariant measure $\mu$ is included in $D_0$, so that 
$$ 
\|\sqrt{\kappa \left( |\nabla u|\right)} \,\partial_{x_i}
\partial_{x_j} u\|^2_H 
$$
is well-defined (and finite) $\mu$-a.e., and 
\begin{equation} 
\label{FiniteMoment} 
\sum_{i,j=1}^d \int_H \|\sqrt{\kappa \left( |\nabla u|\right)} \,\partial_{x_i}
\partial_{x_j} u\|^2_H \, \mu (du) < +\infty . 
\end{equation} 
The Kolmogorov operator associated with \eqref{eq:main-spde} on sufficiently smooth functions $F$ is given by  
\[
J_{0}F(u):= \frac{1}{2}\operatorname{Tr}_H\left( BB^\ast D^{2}F(u)\right) 
- \langle A(u),DF(u)\rangle .
\]
In order to realize $J_0$ as an operator in $L^{1}(H,\mu)$ with domain $D(J_{0}) := C_{b}^{2}(H)$, we need 
to impose conditions on $A$ ensuring that $\|A(u)\|\in H$. To this end we now pose the following 
additional assumption.

\begin{assumption}
Suppose that
\begin{enumerate}
\item[(C7)] there exists a constant $M>0$ such that for every $r\in\R,$
\[
\left|  \psi'(\abs{r})-\frac{\psi(\abs{r})}{\abs{r}}\right|
+ \frac{\psi(\abs{r})}{\abs{r}}  \le M
\sqrt{\kappa (r)} . 
\]
\end{enumerate}
\end{assumption}

\begin{example}
We shall collect the following examples for (C7) and leave their verification to the reader.
\begin{enumerate}
\item
For $p\in (1,2)$, there exists as constant $C(p)>0$ such that $\psi(r)=(1+|r|^2)^{\frac{p-2}{p}}r$ satisfies (C7) with $M=C(p)$.
\item $\psi(r)=\log(1+|r|)\sgn(r)$ satisfies (C7) with $M=2$.
\item $\psi(r)=\arctan(r)$ satisfies (C7) with $M=4$.
\end{enumerate}
\end{example}

\begin{rem}
We would like to point out that for the singular $p$-Laplace $\psi(r)=|r|^{p-2}r$, (C7) reduces to
\[\sqrt{p-1}|r|^{p-2}\le M |r|^{\frac{p-2}{2}}\quad \text{for every $r\in\R$.}\]
Clearly, the existence of such $M>0$ cannot be proved for $p\not=2$.
Neither does the minimal surface flow satisfy (C7).
\end{rem}

An immediate implication of (C7) is that for $u\in D_0$ due to \eqref{representation:A} 
$$ 
\begin{aligned} 
|A(u)| & \le \left|  \psi'(\abs{\nabla u})-\frac{\psi(\abs{\nabla u})}{\abs{\nabla u}}\right|
\left( \sum_{i,j=1}^d \left( \partial_{x_i}\partial_{x_j} u \right)^2 \right)^{\frac 12} 
+  \frac{\psi(\abs{\nabla u})}{\abs{\nabla u}} \abs{\Delta u}  \\ 
& \le M \sqrt{\kappa  \left( |\nabla u|\right)} \left( \sum_{i,j=1}^d \left( \partial_{x_i}
\partial_{x_j} u \right)^2 \right)^{\frac 12} 
\end{aligned} 
$$
for some finite constant. Moreover, \eqref{FiniteMoment} now implies that $\|A(u)\| \in L^2 (H, \mu )$, 
so that for $F\in C_b^2 (H)$ it follows that $\langle A(u) ,  DF(u)\rangle\in L^2 (H, \mu )$, hence $J_0 F(u)
\in L^1 (H, \mu )$, too. From It\^{o}'s formula for applied to $F(u(t))$ with regular initial condition, it follows that $\mu$ is infinitesimally invariant for $J_{0}$. Therefore, since 
\[
J_{0}F^{2}=2FJ_{0}F+\langle BB^\ast DF,DF\rangle,\quad F\in D(J_{0}),
\]
also,
\[
\int_{H}FJ_{0}F\, d\mu = -\frac{1}{2}\int_{H} \lVert B^\ast DF\rVert^{2}\, d\mu .
\] 
Thus $J_{0}$ is dissipative in the Hilbert space $L^{2}(H,\mu)$.
By similar arguments as in \cite{ESS2008} one can prove that $J_{0}$
is also dissipative in $L^{1}(H,\mu)$. As a consequence, it is closable
and its closure $J:=\bar{J_{0}}$ with the domain $D(J)$ is dissipative.

\begin{thm}[compare with Theorem 4.1 in \cite{ESvRS}]\label{thm:max} The operator
$(J,D(J))$ generates a $C_{0}$-semigroup of contractions on $L^{1}(H,\mu)$.
\end{thm} 

\begin{proof} 
We shall prove that for $\lambda > 0$, $\operatorname{range}(\lambda-J)$
is dense in $L^{1}(H,\mu)$. To this aim, for $\alpha>0$, consider
the Yosida approximation of $A$ defined by 
\[
A_{\alpha}(u)=A(J_{\alpha}(u)),\quad\text{where\ensuremath{\quad J_{\alpha}(u)  
= (\operatorname{Id} + \alpha A)^{-1}(u)}.}
\]
For the sequence $A_{\alpha}$ we have the following: 
\begin{enumerate}
\item For any $\alpha>0$, $A_{\alpha}$ is dissipative and Lipschitz continuous, 
\item $\lVert A_{\alpha}(u)\rVert_{H}\leq\lVert A(u)\rVert_{H}$ for any
$u\in D(A)$. 
\end{enumerate}
 Note that the function $A_{\alpha}$ is not differentiable in general.
Therefore we consider a $C^{1}$-approximation. 
For $\alpha,\beta>0$
we set 
\[
A_{\alpha,\beta}(u):=\int_{H}e^{\beta\Delta} A_{\alpha}(e^{\beta\Delta}u+v)\,\mathcal{N}_{0,\sigma_{\beta}}(dv)
\]
where $\mathcal{N}_{0,\sigma_{\beta}}$ is the Gaussian measure on
$H$ with mean $0$ and covariance operator $\sigma_{\beta}:=\int_{0}^{\beta}e^{2\tau\Delta}\,d\tau$, and $e^{\tau \Delta}$, $\tau > 0$, denotes the semigroup generated by the Dirichlet Laplace operator on $L^2 (\mathcal O)$.
Then, $A_{\alpha,\beta}$ is dissipative since
\[
\langle A_{\alpha,\beta}(u)-A_{\alpha,\beta}(v),u-v\rangle 
= \int_{H}\left\langle A_{\alpha}(e^{\beta\Delta}u+w)-A_{\alpha}(e^{\beta\Delta}v+w), 
e^{\beta\Delta}(u-v)\right\rangle  
\,\mathcal{N}_{0,\sigma_{\beta}}(dw)\ge 0.
\]
We would like to use the Cameron-Martin theorem to show that $A_{\alpha,\beta}$
is differentiable. To this end we first need to check that $e^{\beta\Delta}u\in\operatorname{range}
(\sigma_{\beta}^{\frac 12})$. According to Proposition B.1 in \cite{DaPrZ2nd} it suffices to prove 
that there exists a constant $M$ such that $\|e^{\beta \Delta } v\|^2_H 
\le M \|\sigma_\beta^{\frac 12} v\|_H^2$ 
for all $v\in H$. But this follows from the fact that 
$$ 
\begin{aligned} 
\|e^{\beta \Delta} v\|_H^2 
& = \frac 1\beta \int_0^\beta \| e^{(\beta -\tau)\Delta} e^{\tau\Delta}v\|_H^2 \, d\tau  
\le \frac 1\beta \int_0^\beta \|e^{\tau\Delta}v\|_H^2 \, d\tau \\ 
& = \frac 1\beta \int_0^\beta \langle e^{2\tau \Delta} v,v\rangle \, d\tau 
= \frac 1\beta \langle \sigma_\beta v,v \rangle 
= \frac 1\beta \|\sigma_\beta^{\frac 12} v\|_H^2 .
\end{aligned} 
$$ 
An application of the Cameron-Martin theorem yields 
\[
\langle DA_{\alpha,\beta}(u),h\rangle 
= \int_{H}e^{\beta\Delta}A_{\alpha}(e^{\beta\Delta}u+v)\langle\sigma_{\beta}^{-1}e^{\beta\Delta}h,v\rangle\,\mathcal{N}_{0,\sigma_{\beta}}(dv).
\]
and by the Cameron-Martin formula it is $C^{\infty}$ differentiable.
Moreover, as $\alpha,\beta\to0,A_{\alpha,\beta}\to A$ pointwise.
Let us introduce the following approximating equation
\begin{align}
du_{\alpha,\beta}(t) = -A_{\alpha,\beta}(u_{\alpha,\beta}(t))\,dt + B\,dW_{t}, 
\quad t>0,\quad u_{\alpha,\beta}(0)=u. 
\label{yosida_approx}
\end{align}
Since $A_{\alpha,\beta}$ is globally Lipschitz, equation \eqref{yosida_approx}
has a unique strong solution $(u_{\alpha,\beta}(t))_{t\geq0}$. Moreover,
by the regularity of $A_{\alpha,\beta}$ the process $(u_{\alpha,\beta}(t))_{t\geq0}$
is differentiable on $H$. For any $h\in H$ we set $\eta_{h}(t,u):=Du_{\alpha,\beta}(t,u)\cdot h$.
It holds that 
\begin{align}
\frac{d}{dt}\eta_{h}(t,u) 
= - DA_{\alpha,\beta}(u_{\alpha,\beta}(t,u))\cdot\eta_{h}(t,u),\quad\eta_{h}(0,u)=h. 
\label{yosida_diff}
\end{align}
From the dissipativity of $A_{\alpha,\beta}$, we have that 
\[
\langle DA_{\alpha,\beta}(z)h,h\rangle\geq 0,\quad h\in H,z\in D(A).
\]
Hence by multiplying both sides of \ref{yosida_diff} with $\eta_{h}(t,u)$
and integrating with respect to $t$, we get 
\begin{equation}
\lVert\eta_{h}(t,u)\rVert^{2}\leq\lVert h\rVert^{2}.\label{yosida_diff_norm}
\end{equation}
Now, for $\lambda>0$ and $F\in C_{b}^{2}(H)$ consider the following
elliptic equation 
\begin{equation}
(\lambda-J_{A_{\alpha,\beta}})\varphi_{\alpha,\beta}=F,\ \lambda>0,\label{elliptic}
\end{equation}
where $J_{A_{\alpha,\beta}}$ is the Kolmogorov operator corresponding
to the stochastic differential equation \eqref{yosida_approx}. It
is well-known that this equation has a solution $\varphi_{\alpha,\beta}\in C_{b}^{2}(H)$
and can be written in the form $\varphi_{\alpha,\beta}=R(\lambda,J_{A_{\alpha,\beta}})F$,
where 
\[
R(\lambda,J_{A_{\alpha,\beta}})F(u)=\int_{0}^{+\infty}e^{-\lambda t}\E(F(u_{\alpha,\beta})(t,u))\,dt
\]
is the pseudo resolvent associated with $J_{A_{\alpha,\beta}}$. Thus
we have 
\begin{equation}
\lVert\lambda\varphi_{\alpha,\beta}\rVert_{\infty}\leq\lVert f\rVert_{\infty}.\label{norm}
\end{equation}
We have, moreover, for all $h\in H$,
\[
D{\varphi_{\alpha,\beta}}(u)h 
= \int_{0}^{+\infty}e^{-\lambda t}\E\left(DF(u_{\alpha,\beta})(t,u)((Du_{\alpha,\beta})(t,u)h)\right)dt.
\]
Consequently, using \eqref{yosida_diff_norm}, it follows that 
\[
\lVert D\varphi_{\alpha,\beta}\rVert\leq\frac{1}{\lambda}\lVert DF\rVert_{\infty}.
\]
From \eqref{elliptic}, we have 
\begin{align*}
\lambda\varphi_{\alpha,\beta}(u)  
& -\frac{1}{2}\operatorname{Tr}_H\left( BB^\ast  D^{2}\varphi(u)\right) 
 +\langle A(u),D\varphi_{\alpha,\beta}(u)\rangle\\
 & = F(u) + \langle A(u)-A_{\alpha,\beta}(u),D\varphi_{\alpha,\beta}(u)\rangle,\ \lambda>0,\ u\in D(A).
\end{align*}
Using gradient bound \eqref{norm}, we deduce that for $\alpha,\beta\to0$,
\[
\int_{H}\lvert\langle A(u)-A_{\alpha,\beta}(u),D\varphi_{\alpha,\beta}(u)\rangle\rvert\,\mu(du) 
\leq\frac{1}{\lambda}\lVert DF\rVert_\infty \lVert A_{\alpha,\beta}-A\rVert_{L^{2}(H,\mu)}.
\]
By Lebesgue's theorem $\lVert A_{\alpha,\beta}-A\rVert_{L^{2}(H,\mu)}$
converges to $0$ as $\alpha,\beta\to0$. Therefore we deduce that
for $\alpha,\beta\to0$ 
\[
\lambda\varphi_{\alpha,\beta}(u) 
- \frac{1}{2}\operatorname{Tr}_H \left( BB^\ast D^{2}\varphi_{\alpha,\beta} (u)\right) 
+ \langle A(u),D\varphi_{\alpha,\beta}(u)\rangle\to F
\]
strongly in $L^{1}(H,\mu)$. This implies that 
\[
C_{b}^{2}(H)\subset\overline{(\lambda-J_{0})(D(J_{0}))}.
\]
Since $C_{b}^{2}(H)$ is dense in $L^{1}(H,\mu)$, the proof is complete.\end{proof}

\section*{Statements and Declarations}

The authors have no relevant financial or non-financial interests to disclose.
The authors have no competing interests to declare that are relevant to the content of this article.

The work of the second author was partially supported by the German Science Foundation (DFG) via Research Unit FOR 2402 (grant no. 277012070).

The third author acknowledges support by the Academy of Finland and the European Research Council (ERC) under the European Union's Horizon 2020 research and innovation programme (grant agreements no. 741487 and no. 818437).

\end{document}